\documentclass[12pt]{amsart}
\usepackage{amssymb}
\usepackage{amscd}
\usepackage{color}

\newcommand{\Diff}{\mathcal{D}}
\newcommand{\Lie}{\mathcal{L}}

\newcommand{\id}{\text{id}}

\newcommand{\grad}{\nabla}
\newcommand{\Laplacian}{\Delta}
\newcommand{\bundleproj}{\overline{P_{\theta}}}
\DeclareMathOperator{\ad}{ad} \DeclareMathOperator{\Ad}{Ad}

\DeclareMathOperator{\diver}{div} 
\newcommand{\Exp}{\text{Exp}}
\newcommand{\Reeb}{E}
\newcommand{\adj}{\star}
\DeclareMathOperator{\ann}{ann}
\newcommand{\pushedvolume}{\tilde{\mu}}
\newcommand{\mingeodesic}{\chi}
\newcommand{\contact}{\theta}
\newcommand{\Diffcontact}{\Diff_{\contact}}

\newcommand{\Diffexcontact}{\Diff_{q}}
\newcommand{\Diffsexcontact}{\Diff^s_q}

\newcommand{\Diffximu}{\Diff_{\Reeb, \mu}}
\newcommand{\Diffham}{\Diff_{\text{Ham}}}
\newcommand{\Diffxi}{\Diff_{\Reeb}}
\newcommand{\Diffsxi}{\Diff^s_{\Reeb}}
\newcommand{\Diffsmu}{\Diff^s_{\mu}}
\newcommand{\Diffmu}{\Diff_{\mu}}
\newcommand{\Func}{\mathcal{F}}
\newcommand{\llangle}{\langle\!\langle}
\newcommand{\rrangle}{\rangle\!\rangle}
\newcommand{\Llangle}{\Big\langle\!\!\Big\langle}
\newcommand{\Rrangle}{\Big\rangle\!\!\Big\rangle}

\newcommand{\betaimage}{T_{\contact}\Lambda^1_{\Reeb,\mu}}
\newcommand{\pullbackimage}{\Lambda_{\Reeb}}
\newcommand{\imcont}{\alpha}
\newcommand{\imcontvar}{\beta}
\newcommand{\pullback}{\mathcal{P}}

\newtheorem{theorem}{Theorem}[section]
\newtheorem{proposition}[theorem]{Proposition}
\newtheorem{lemma}[theorem]{Lemma}
\newtheorem{corollary}[theorem]{Corollary}

\theoremstyle{definition}
\newtheorem{defn}[theorem]{Definition}
\newtheorem{exmp}[theorem]{Example}

\theoremstyle{remark}
\newtheorem{rem}[theorem]{Remark}

\begin{document}

\title[]{Riemannian geometry of the quantomorphism group}

\author{David G. Ebin}
\address{Department of Mathematics, Stony Brook University, Stony Brook, NY 11794-3651}
\email{ebin@math.sunysb.edu}
\author{Stephen C. Preston}
\address{Department of Mathematics, University of Colorado,
Boulder, CO 80309-0395} \email{stephen.preston@colorado.edu}
\date{\today}

\begin{abstract}
We are interested in the geometry of the group $\Diffexcontact(M)$ of diffeomorphisms preserving a contact form $\contact$ on a manifold $M$. We define a Riemannian metric on $\Diffexcontact(M)$, compute the corresponding geodesic equation, and show that solutions exist for all time and depend smoothly on initial conditions. In certain special cases (such as on the $3$-sphere), the geodesic equation is a simplified version of the quasigeostrophic equation,
so we obtain a new geodesic interpretation of this geophysical system. 
We also show that the genuine quasigeostrophic equation on $S^2$ can be obtained as an Euler-Arnold equation on a one-dimensional central extension of $T_{\id}\Diffexcontact(M)$, and that our global existence result extends to this case.

If $\Reeb$ is the Reeb field of $\contact$ 
and $\mu$ is the volume form, assumed compatible in the sense that $\diver{\Reeb}=0$, we show that $\Diffexcontact(M)$ is a smooth submanifold of $\Diffximu(M)$, the space of diffeomorphisms preserving the vector field $\Reeb$ and the volume form $\mu$, in the sense of $H^s$ Sobolev completions. 
The latter manifold is related to symmetric motion of ideal fluids. We further prove that the corresponding geodesic equations and projections are $C^{\infty}$ objects in the Sobolev topology.
\end{abstract}

\maketitle

\section{Introduction and Background}

Contact geometry is the odd-dimensional analogue of symplectic geometry~\cite{EKM,geiges}, and the purpose
of this article is to extend some results on the Riemannian geometry of symplectic forms~\cite{ebinsymplectic} to 
the contact case. On a manifold $M$ of dimension $2n+1$, a \emph{contact structure} is a distribution of $2n$-planes
which \emph{maximally} does not satisfy the Frobenius integrability condition at any point. If the contact structure is orientable (as in typical examples), it is determined as the null space of a \emph{contact form}
$\contact$, a $1$-form on $M$ for which maximal nonintegrability translates into 
\begin{equation}\label{contactnonintegrable}
\contact\wedge (d\contact)^n \ne 0
\end{equation} everywhere. 
On $\mathbb{R}^{2n+1}$ with coordinates $(x^1,\ldots, x^n, y^1,\ldots, y^n, z)$,
the basic example is given by
\begin{equation}\label{darbouxcontact}
\contact = dz + \sum_{k=1}^n x^k \, dy^k.
\end{equation}
By the Darboux theorem~\cite{geiges}, we can always find local coordinates on any manifold $M$ such that $\contact$ is given by \eqref{darbouxcontact}.

It is important to note that if $\contact$ is a contact form, then for any nowhere zero function $F$ on $M$, the $1$-form $F\contact$ also satisfies \eqref{contactnonintegrable} and its null space is the same contact structure. Hence if we care primarily about the contact \emph{structure}, the symmetries are the group $\Diffcontact(M)$ of diffeomorphisms $\eta$ such that $\eta^*\contact = F\contact$ for some positive $F$; if we care primarily about the contact \emph{form}, the symmetries are the group $\Diffexcontact(M)$ of diffeomorphisms $\eta$ with $\eta^*\contact=\contact$. In this article we will work with the latter group, which following Ratiu-Schmid~\cite{RS}, we call the \emph{quantomorphism group}. In a separate article~\cite{ebinprestoncontacto} we will study the geometry of the contactomorphism group $\Diffcontact(M)$, which also has a number of interesting features and has not been studied from the Riemannian point of view at all.

The quantomorphism group was studied by Omori~\cite{omori} (as an ILH group), Ratiu-Schmid~\cite{RS} (as a Hilbert manifold), and Smolentsev~\cite{smolentsevclassical} (as a formal Riemannian manifold). 
More recently several authors have observed that the quantomorphism group is the proper configuration 
space for the geodesic Vlasov equations (our geodesic equation is essentially a special case of this). 
Holm-Tronci~\cite{HT} proposed a general quadratic Hamiltonian
on the quantomorphism group and studied its moments. Gay-Balmaz and Vizman~\cite{GV} computed the momentum maps for this equation, and Gay-Balmaz and Tronci~\cite{GT} found an interesting totally geodesic subgroup. We would like to thank Darryl Holm, Tudor Ratiu, and Fran\c cois Gay-Balmaz for very helpful discussions. Our main contribution here is local and global well-posedness of the geodesic equation in Sobolev spaces.

We now review some well-known properties of contact structures.

\begin{proposition}\label{gammaoperatordef}
There is a Reeb vector field $\Reeb$ which is uniquely defined in each tangent space by the conditions $\contact(\Reeb)=1$ and $\iota_{\Reeb}d\contact = 0$. 
The linear map $u\mapsto \iota_ud\contact$ from $T_xM$ to $T_x^*M$
yields a pointwise isomorphism $\gamma$ from the null space of $\contact$ to the annihilator of $\Reeb$.
\end{proposition}

\begin{proof}
In Darboux coordinates we have $d\contact = \sum_{k=1}^n dx^k \wedge dy^k$, and 
it is easy to check that the Reeb field must be $\Reeb = \partial_z$. The null space of $\contact$ is spanned by the basis $\{ \partial_{x^k}, \partial_{y^k}-x^k \, \partial_z : 1\le k\le n\}$, and we have $\gamma(\partial_{x^k}) = dy^k$ and $\gamma(\partial_{y^k}-x^k \, \partial_z) = -dx^k$ for all $k$. Since $dy^k(\Reeb) = dx^k(\Reeb)=0$, the map $\gamma$ takes a basis of vectors annihilated by $\contact$ to a basis of covectors that annihilate $\Reeb$. Since these definitions are coordinate-invariant, they define global objects $\Reeb$ and $\gamma$.
\end{proof}

Properties of the Reeb field determine the structure of the quantomorphism group, in the following way.

\begin{proposition}\label{reebpreserving}
Any diffeomorphism $\eta$ for which $\eta^*\contact = \contact$ must also satisfy $\eta_*\Reeb = \Reeb$. As a result the quantomorphism group $\Diffexcontact(M)$ is a subgroup of $\Diffxi(M)$, the group of diffeomorphisms preserving $\Reeb$ (alternatively, the diffeomorphisms commuting with the flow of $\Reeb$).
\end{proposition}

\begin{proof}
We want to show that $\eta^*\contact=\contact$ implies that $D\eta \circ\Reeb = \Reeb\circ\eta$. We write $\eta_*\Reeb = D\eta\circ\Reeb\circ\eta^{-1}$ so that $\eta_*\Reeb$ is another vector field on $M$. Then we have
$$ 
\iota_{\eta_*\Reeb}d\contact = \iota_{\Reeb}(\eta^*d\contact) \circ \eta^{-1} = \iota_{\Reeb}(d\contact) \circ\eta^{-1} = 0,
$$
along with $$\contact(\eta^*\Reeb) = \eta^*\contact(\Reeb)\circ\eta^{-1} = 1\circ\eta^{-1}=1.$$
Since $\eta_{*}\Reeb$ and $\Reeb$ satisfy the same conditions, they must be equal.
\end{proof}

If $\mu$ is the Riemannian volume form, and we assume (as we shall always do) that $\mu$ is a constant multiple of $\contact \wedge (d\contact)^n$, then clearly $\eta^*\contact = \contact$ implies that $\eta^*\mu = \mu$. Hence Proposition \ref{reebpreserving} shows that $\Diffexcontact(M)$ is a subgroup of $\Diffximu(M)$, which we refer to as the group of \emph{volumorphisms with symmetry}.

\begin{exmp}\label{hopffibration}
Our most important example will be $S^3$, considered as a group of unit quaternions in $\mathbb{R}^4$. It has a global basis of left-invariant vector fields $\{E_1, E_2, E_3\}$ satisfying the relations 
$$[E_1, E_2] = -2E_3, \qquad [E_2, E_3] = -2E_1, \quad\text{and}\quad [E_3,E_1] = -2E_2.$$ Letting $\{\omega^1, \omega^2, \omega^3\}$ denote the dual basis, a contact form is given by $\contact = \omega^1$, with Reeb field $\Reeb = E_1$. We have $d\contact = 2 \omega^2\wedge \omega^3$, so that $\contact\wedge d\contact$ is twice the Riemannian volume form. If we declare $\{E_i\}$ to be orthongal with  $\lVert E_1\rVert = \alpha$ and $\lVert E_2\rVert = \lVert E_3\rVert = 1$ for some parameter $\alpha>0$, then we have the Berger metric on $S^3$. (See for example \cite{petersen}.)

The Reeb field $\Reeb=E_1$ is Killing, and all orbits are closed and have the same length. The quotient is $S^2$, and the projection $\pi\colon S^3\to S^2$ is the well-known Hopf fibration. The pullback of the standard area form on $S^2$ is a constant multiple of $d\contact$. Any quantomorphism $\eta$ on $S^3$ generates a symplectomorphism $\zeta$ on $S^2$ by the formula $\zeta(p) = \pi(\eta(\pi^{-1}(p)))$, which is well-defined since $\eta$ commutes with the flow of $E_1$ and hence is constant on the fibers of $\pi$. Notice that the flow of $E_1$ gives a family of quantomorphisms which preserve the fibers of $\pi$, and hence all map to the identity symplectomorphism of $S^2$. This is the typical behavior we expect quantomorphisms to have; see below.
\end{exmp}

Not every contact manifold has an interesting quantomorphism group. 

\begin{exmp}
On $M=\mathbb{T}^3$, we can check that $\contact = \sin{z} \, dx + \cos{z} \, dy$ satisfies $\contact\wedge d\contact = dx\wedge dy\wedge dz$, so that $\contact$ is a contact form. The Reeb field is $\Reeb=\sin{z} \, \partial_x + \cos{z} \, \partial_y$. Every quantomorphism must preserve the Reeb field, but the Reeb field has nonclosed orbits whenever $\tan{z}$ is irrational, and hence any function which is constant on the orbits must actually be a function only of $z$. It is then easy to see that the identity component of $\Diffexcontact(\mathbb{T}^3)$ consists of diffeomorphisms of the form
$$ \eta(x,y,z) = \big(x+p(z)\sin{z} + p'(z) \cos{z}, y+p(z)\cos{z}-p'(z)\sin{z}, z\big)$$
for some function $p\colon S^1\to \mathbb{R}$. This group is abelian, so any right-invariant metric will actually be bi-invariant, and all geodesics will be one-parameter subgroups. 
\end{exmp}


Instead what we want is for all the orbits of the Reeb field $\Reeb$ to be closed and of the same period, so that the flow $t\mapsto \varphi(t,x)$ of $\Reeb$ is periodic.
In this case we say that the vector field and the corresponding contact form are \emph{regular}~\cite{RS}. When this happens, there is a Boothby-Wang fibration~\cite{BW} $\pi\colon M\to N$, where $N$ is a symplectic manifold of dimension $2n$ with symplectic form $\omega$ satisfying $\pi^*\omega=d\contact$. Using Darboux coordinates $(x^k, y^k, z)$ on $M$, we get coordinates $(x^k, y^k)$ on $N$ such that $\omega = \sum_k dx^k\wedge dy^k$. The Hopf fibration of Example \ref{hopffibration} is the basic example.

\begin{exmp}
To get a Boothby-Wang fibration, it is important that the orbits are not only closed but of the same length. Here we give an example of what happens when this fails.
The Hopf fibration is commonly constructed by considering $S^3$ as the unit sphere in $\mathbb{C}^2.$ Using complex coordinates $(w,z)$ we let the circle action be $\tau \mapsto (e^{i\tau} w, e^{i\tau} z)$ so that all orbits are closed and of period $2\pi.$  A variation of this, which also comes from a contact form, is the action  $\tau \mapsto (e^{ik\tau} w, e^{il\tau} z)$
where $k$ and $l$ are relatively prime.  In this case the orbits have period $2\pi$ except when $w$ or $z$ is zero in which case they have period of $2\pi/k$ or $2\pi/l$.  Thus the quotient $S^3/S^1$ has singularities and is not a manifold at the images of these points.  This construction actually comes from a simple problem in mechanics.  Take two harmonic oscillators with periods $2\pi/k$ and $2\pi/l$ respectively.  Each of their motions is a curve in $\mathbb{C}$ so together they are in $\mathbb{C}^2.$  $S^3 \subset \mathbb{C}^2$ will then be a constant energy surface and the motion on it will be the $S^1$ action described above.  We thank H. Hofer (personal communication) for this example.
\end{exmp}

Let us now review some facts about the group of quantomorphisms. At first we will  assume that all objects are $C^{\infty}$; later we will work with spaces of Sobolev diffeomorphisms, so that we can do analysis on Hilbert manifolds and use tools of differential analysis such as the inplicit function theorem.

\begin{proposition}\label{quantotangentprop}
The tangent space to the quantomorphism group 
$ \Diffexcontact(M) = \{\eta\in \Diff(M)\, \big\vert \, \eta^*\contact = \contact\}$ at the identity is 
\begin{equation}\label{quantotangent}
T_{\id}\Diffexcontact = \big\{ S_{\contact}f \, \vert \,  
f\in \Func_E(M,\mathbb{R})\},
\end{equation}
where 
\begin{equation}\label{functionEspace}
\Func_E(M,\mathbb{R}) = \{ f\colon M\to \mathbb{R} \,\big\vert\, 
\Reeb(f)\equiv 0\big\}
\end{equation}
and the operator $S_{\contact}$ is defined by the conditions 
\begin{equation}\label{Salphadef}
u=S_{\contact}f \quad \Longleftrightarrow \quad  \contact(u) = f \quad \text{and}\quad i_u d\contact = -df.
\end{equation}
For any $f$ with $\Reeb(f)\equiv 0$, the field $u=S_{\contact}f$ satisfies  $\Lie_u\contact = 0$; conversely if $\Lie_u\contact = 0$, then $u=S_{\contact}f$ for some $f$ with $\Reeb(f)\equiv 0$.
\end{proposition}

\begin{proof}
Since $\eta\in \Diffexcontact(M)$ if and only if $\eta^*\contact=\contact$, a tangent vector $u$ at the identity must satisfy $\Lie_u\contact=0$, or $d(\contact(u)) + \iota_ud\contact=0$. Letting $f=\contact(u)$, we have $\iota_ud\contact = -df$. Applying both sides to the Reeb field $\Reeb$, we have $\Reeb(f) = d\contact(\Reeb,u) = 0$. Since $\Reeb(f)=0$, we know $df$ is in the annihilator of $\Reeb$, and since
 $\gamma$ is an isomorphism from the null space of $\contact$ to the annihilator of $\Reeb$, we get a unique field $v$ for which $\contact(v)\equiv 0$ and $\gamma(v) = -df$. Then we must have $u=v+fE$ in order to satisfy both $\contact(u)=f$ and $\gamma(u)= -df$. The rest of the proof is straightforward.
\end{proof}

Note that any $C^{\infty}$ function $f\colon M\to \mathbb{R}$ satisfying $\Reeb(f)\equiv 0$ is constant on orbits of $E$ and thus may be viewed as a function $\tilde{f}\colon N\to \mathbb{R}$ where $N$ is the symplectic quotient. We can thus identify the space $\Func_{\Reeb}(M,\mathbb{R})$ from \eqref{functionEspace} with $\Func(N, \mathbb{R})$ whenever convenient. Also since $\omega$ is a symplectic form on $N$ we get an isomorphism $\omega^{\flat}:TN \to T^* N$ defined by
$\omega^{\flat}(v) = \iota_v \omega.$  We denote its inverse by $\omega^{\sharp}$.

Using $S_{\contact}$ we can define a type of Poisson bracket on $\Func_{\Reeb}(M,\mathbb{R}).$  We define
\begin{equation}\label{contactbracket}
\{f,g\}:= S_{\theta} f(g).
\end{equation}
Thus in Darboux coordinates we have 
\begin{equation}\label{ScontactDarboux}
S_{\contact}f = \sum_{k=1}^n \Big( 
-\frac{\partial f}{\partial y^k} \, \frac{\partial}{\partial x^k} + 
\frac{\partial f}{\partial x^k} \, \frac{\partial}{\partial y^k}\Big)
+ \Big( f - \sum_{k=1}^n x^k \,\frac{\partial f}{\partial x^k}\Big) \, \frac{\partial}{\partial z}
\end{equation}
and 
\begin{equation}\label{bracketDarboux} \{f,g\} = \sum_{k=1}^n \frac{\partial f}{\partial x^k} \, \frac{\partial g}{\partial y^k} - \frac{\partial f}{\partial y^k} \, \frac{\partial g}{\partial x^k}
\end{equation}
(since of course $\frac{\partial g}{\partial z}=0$ for $g\in \Func_{\Reeb}(M,\mathbb{R})$). 
%
%

\begin{proposition}\label{bracketisomorphism}
$S_{\theta}:\Func_{\Reeb}(M,\mathbb{R})\rightarrow
T_{\id}\Diffexcontact$ is an algebra isomorphism; that is, it takes Poisson brackets into Lie brackets.
\end{proposition}

\begin{proof} Let $u=S_{\theta}f$ and $v = S_{\theta}g.$  Then since
$S_{\theta}$ is surjective, there exists $h \in \Func_{\Reeb}(M,\mathbb{R})$ such that $S_{\theta}h = [u,v].$
Then $\theta(S_{\theta}h) = h$ and
\begin{eqnarray*}
\theta([S_{\theta}f, S_{\theta}g]) & = & -d\theta(S_{\theta}f, S_{\theta}g) + S_{\theta}f\big(\theta(S_{\theta}g)\big) - S_{\theta}g\big(\theta(S_{\theta}f)\big)\\
& = & -\iota_{S_{\theta}f} d\theta(S_{\theta}g)+ S_{\theta}f(g) -S_{\theta}g(f)\\
& = & df(S_{\theta}g) +\{f,g\} -df(S_{\theta}g) \\
& = & \{f,g\} \end{eqnarray*}
Thus $h =\{f,g\}$ so $ S_{\theta}(\{f,g\}) = [S_{\theta}f,S_{\theta}g]$.
\end{proof}

Note that $\pi_* (S_{\theta} f) = -\omega^{\sharp}(d\tilde{f})$, and the identification of $\Func_{\Reeb}(M,\mathbb{R})$ with $\Func(N, \mathbb{R})$ takes the Poisson bracket of $\Func_{\Reeb}(M,\mathbb{R})$ into the Poisson bracket of $\Func(N, \mathbb{R})$.  Also since the standard bracket has the property $$\int_N \{\tilde{f},\tilde{g}\} \tilde{h} \, d\nu = \int_N \{\tilde{h},\tilde{f}\}\tilde{g} \,d\nu $$
where $\nu$ is the symplectic volume form on $M$,
we get $$\int_M \{f,g\}h \, d\mu= \int_M \{h,f\}g\,d\mu$$ as well, where $\mu$ is the contact volume form.  
This fact will be needed for the proof of Theorem \ref{localgeodesic}.

If we denote the operator $u\mapsto \Lie_u\contact$ by $L_{\contact}$, then by construction we have $L_{\contact}\circ S_{\contact}=0$. We can then view the operators $S_{\contact}$ and $L_{\contact}$ as forming a short exact sequence if we choose their domains and ranges correctly. This will be important when we discuss smoothness of the geodesic equation, so we record the formal result.

\begin{proposition}\label{exactsequenceprop}
Let $\Func_{\Reeb}(M,\mathbb{R})$ denote the space of smooth functions $f\colon M\to \mathbb{R}$ such that $\Reeb(f)\equiv 0$ as in \eqref{functionEspace}. Let $\betaimage(M)$ denote the space of one-forms $\beta$ on $M$ such that  $\iota_{\Reeb}\beta=0,$ $\iota_{\Reeb}d\beta=0$ and $\contact\wedge d\beta\wedge (d\contact)^{n-1}=0$. Finally, let $T_{\id}\Diffximu(M)$ denote the tangent space of $\Diffximu(M)$, which consists of vector fields $u$ such  that $\diver{u}=0$ and $[\Reeb,u]=0$. If $L_{\contact}u = \Lie_u\contact$ and $S_{\contact}$ is defined as in \eqref{Salphadef}, then 
\begin{equation}\label{shortexactSalpha}
0 \to \Func_{\Reeb}(M,\mathbb{R})\stackrel{S_\theta}{ \to} T_{\id}\Diffximu(M) \stackrel{L_\theta}{\to}\betaimage(M) \to 0
\end{equation}
is a short exact sequence, i.e., the image of every map is the null space of the next map.
\end{proposition}

\begin{proof}
That $S_{\contact}$ is one-to-one from $\Func_{\Reeb}(M,\mathbb{R})$ is easy to check. The fact that $S_{\contact}$ maps into $T_{\id}\Diffximu$ follows from 
Proposition \ref{reebpreserving}, as well as the fact that the volume form $\mu$ is a multiple of $\contact\wedge (d\contact)^n$.
To show exactness at $T_{\id}\Diffximu(M)$ we pick $u \in T_{\id}\Diffximu(M).$
$L_\theta u=0$ if and only if $i_u d\theta = -d\theta(u)$.  Letting $f=\theta(u),$
we find that $u = S_\theta f$ so exactness here follows.   

Finally the fact that $L_{\contact}$ maps $T_{\id}\Diffximu(M)$ onto $\betaimage(M)$ follows from the following computation. Let $\beta\in \betaimage(M)$, and define $u$ to be the unique vector field satisfying $\contact(u)=0$ and $\gamma(u)=\beta$. Then $L_{\contact}(u) = \beta$, and we just need to check that $u\in T_{\id}\Diffximu(M)$. We have  $$\Lie_{[\Reeb, u]}\contact = \Lie_{\Reeb}\Lie_u\contact - \Lie_u\Lie_{\Reeb}\contact = \Lie_{\Reeb}\beta = 0,$$
and since $\contact([\Reeb,u]) = \Reeb(\contact(u)) - u(\contact(\Reeb)) - d\contact(\Reeb,u) = 0$, we conclude that $\gamma([\Reeb,u])=0$ as well. Since the only vector field $v$ satisfying $\contact(v)=0$ and $\gamma(v)=0$ is zero, we conclude $[\Reeb,u]=0$. The fact that $u$ is divergence-free follows from the computation
$$ \Lie_u\big(\contact\wedge (d\contact)^n\big) = \Lie_u\contact \wedge (d\contact)^n + n\contact \wedge d\Lie_u\contact \wedge (d\contact)^{n-1} = 0$$
since $\Lie_u\contact=\beta$.
\end{proof}

\section{Sobolev manifold structures}

In this section we discuss the Riemannian metric on $\Diffexcontact$ and its geodesic equation. We will show that $\Diffexcontact$ is a smooth submanifold of $\Diffximu$, and that the geodesic equation is a smooth ordinary differential equation on $\Diffexcontact$.
The results of this section generally do not depend on the choice of Riemannian metric on $M$; we can always arrange things so that there is some metric on $M$ which makes $\Reeb$ a Killing field, by the regularity assumption on $\Reeb$.

To discuss smoothness, we extend our spaces of $C^{\infty}$ maps to Sobolev $H^s$ maps, where $s$ is an integer strictly larger than $\dim{M}/2+1 = n+\frac{3}{2}$ (in order to ensure that any such map is $C^1$). These structures have previously been studied in this context by Omori~\cite{omori}, Ratiu-Schmid~\cite{RS}, and Smolentsev~\cite{smolentsev}. Our interest is in the quantomorphism group $\Diffexcontact^s(M)$ as a submanifold of the volumorphism group with symmetry, $\Diffximu^s(M)$, which in turn is a submanifold of $\Diffmu^s(M)$ (which is known to be a smooth submanifold of $\Diff^s(M)$ by \cite{EM}). 

We will be dealing with many Hilbert manifolds, all of which are topological subgroups of the full diffeomorphism group $\Diff^s(M)$, and we would like to show they are all submanifolds of this group. This is not automatic even for nice diffeomorphism groups; for example the contactomorphism group $\Diff_{\contact}^s(M)$ is a $C^{\infty}$ Hilbert manifold but not a $C^{\infty}$ Hilbert submanifold of $\Diff^s(M)$ (see Omori~\cite{omori}). We will generally prove our subgroups are submanifolds either by constructing an explicit coordinate chart or by using the implicit function theorem. In order to avoid duplication, we will do this only in the simplest cases, and derive the other cases from it. 

\begin{lemma}\label{3submanifold}
Suppose $L$ and $M$ are both $C^{\infty}$ Hilbert submanifolds of a $C^{\infty}$ Hilbert manifold $N$, and that $L$ is a subset of $M$. Then $L$ is a $C^{\infty}$ Hilbert submanifold of $M$.
\end{lemma}

\begin{proof}
For general properties of Hilbert manifolds, see \cite{AMR} or \cite{lang}. We will use the fact that a subset is locally a smooth submanifold if and only if the inclusion map is a smooth immersion. Let $i\colon L\to M$, $j\colon M\to N$, and $k\colon L\to N$ be the inclusions. We know $j$ and $k$ are smooth immersions and want to prove that $i$ is as well. Clearly $j\circ i = k$, and if we knew that $i$ were smooth, the Chain Rule would imply that it would have to be an immersion since $Tj\circ Ti = Tk$, with $Tj$ and $Tk$ both injective. 

Smoothness of $i$ follows from the ``universal mapping'' property of submanifolds; see Lang~\cite{lang}. More explicitly, we get smoothness using coordinate charts.  Since $M$ is a smooth submanifold of $N$, for every $p\in L\subset M$ there is a chart $\psi\colon W\to G$ from a neighborhood $W$ of $p$ in $N$ to a Hilbert space $G$, such that for some closed subspace $F\subset G$ we have 
$$ \psi[W\cap M] = \psi[W] \cap F\times \{0\},$$
where $\{0\}$ denotes the zero element of the orthogonal complement $F^{\perp}$. The map $\overline{\psi} = \pi_F \circ \psi\vert_{W\cap M} \colon W\cap M\to F$ given by restricting the domain and projecting the range onto $F$ is a coordinate chart on $M$. 

Now since $\psi$ is a coordinate chart on $N$, we know $\psi\circ k\colon L\to G$ is $C^{\infty}$. Thus $\pi_F\circ\psi\circ k\colon L\to F$ is also $C^{\infty}$. But $\pi_F\circ\psi\circ k = \overline{\psi}\circ i$, and since $\overline{\psi}\circ i$ is smooth for a coordinate chart $\overline{\psi}$, the inclusion map $i$ must be smooth by definition.
\end{proof}

To begin we consider the set 
\begin{equation}\label{diffxidef}
\Diffxi^s(M) = \{\eta\in\Diff^s(M) \, \vert \, \eta_*\Reeb = \Reeb \}
\end{equation}
of $H^s$ diffeomorphisms of $M$ which preserve the Reeb field (or equivalently, the set of $H^s$ diffeomorphisms which commute with the flow of the Reeb field). We will prove this is a submanifold of $\Diff^s(M)$, a special case of a result due originally to Omori~\cite{omori}; we provide the proof to make the paper somewhat more self-contained.

\begin{theorem}\label{diffxisubmfd}
Let $\Reeb$ be a vector field on $M$ with closed orbits all of the same period. The set $\Diffxi^s(M)$ of $H^s$ diffeomorphisms preserving $\Reeb$ (or commuting with its flow) is a $C^{\infty}$ submanifold of $\Diff^s(M)$. 
\end{theorem}

\begin{proof}
We may assume without loss of generality that $\Reeb$ is a Killing field of some Riemannian metric on $M$; we can simply take an arbitrary metric and average it over a period of the flow of $\Reeb$. 

The coordinate charts on $\Diff^s(M)$ are defined in terms of the Riemannian metric on $M$ as follows: for any $\eta\in \Diff^s$ we consider the linear space $H^s_{\eta}(TM)$ of $H^s$ vector fields $V$ over $\eta$, i.e., such that $V(x)\in T_{\eta(x)}M$ for every $x\in M$. Then using the Riemannian exponential map $\exp$ on $M$, we define an exponential map $\Exp_{\eta}\colon H^s_{\eta}(TM)\to \Diff^s$ by $ \Exp_{\eta}(V) = x\mapsto \exp_{\eta(x)}(V(x))$ (i.e., for each $x$ we follow the geodesic starting at position $\eta(x)$ with velocity $V(x)$ for time one). If the $H^s$ norm of $V$ is sufficiently small, then by the Sobolev embedding theorem the $C^0$ norm of $V$ will be smaller than the injectivity radius of $M$, and thus $\Exp_{\eta}$ is invertible on a small neighborhood $\Omega$ of zero in $H^s_{\eta}(TM)$; its inverse is the desired chart.

To prove $\Diffxi^s$ is a submanifold, it is sufficient to show that in these charts, for any $\eta\in \Diffxi^s(M)$, we have 
\begin{equation}\label{submanifoldisometrycondition}
\Exp_{\eta}\big[\Omega \cap K^s_{\eta}\big] = \Exp_{\eta}[\Omega] \cap \Diffxi^s
\end{equation}
for some closed subspace $K^s_{\eta}$ of $H^s_{\eta}(TM)$. Set 
$$K^s_{\eta} = \{ v\circ\eta \, \vert \, [\Reeb, v] = 0\} = R_{\eta}[K^s_{\id}],$$
where $R_{\eta}$ is the right-translation. 
Since $K^s_{\id}$ is the null-space of the continuous map $v\mapsto [\Reeb, v]$ from $H^s$ vector fields to $H^{s-1}$ vector fields, it is a closed subspace of $H^s_{\id}(TM)$; since the right translation $R_{\eta}$ is an isomorphism, each space $K^s_{\eta}$ is closed in $H^s_{\eta}(TM)$.

Let $\zeta_{\tau}$ denote the flow of $\Reeb$;
since $\Reeb$ is a Killing field, every $\zeta_{\tau}$ is an isometry. Since isometries preserve geodesics we have, for any $\eta\in \Diff^s$ and $V\in H^s_{\eta}(TM)$, that 
\begin{equation}\label{preservegeodesics}
\exp_{\zeta_{\tau}(\eta(x))}\big((\zeta_{\tau})_*V(x)\big) = \zeta_{\tau}\Big(\exp_{\eta(x)}(V(x))\Big)
\end{equation} 
for all $x\in M$ and $\tau\in \mathbb{R}$.
Now if $V\in K^s_{\eta}$, then $V=v\circ\eta$ for some $v\in K^s_{\id}$, and since $[\Reeb, v]=0$, the flow $\zeta_{\tau}$ of $\Reeb$ preserves $v$, i.e., we have $(\zeta_{\tau})_*v(x) = v\big(\zeta_{\tau}(x))$ for all $x\in M$ and all $\tau\in \mathbb{R}$. We conclude that $(\zeta_{\tau})_*V(x) = v\big( \zeta_{\tau}(\eta(x))\big)$. Combining with \eqref{preservegeodesics} we obtain 
\begin{equation}\label{commutingisometry}
\exp_{\zeta_{\tau}(\eta(x))}\big(v\big(\zeta_{\tau}(\eta(x))\big)\big) = \zeta_{\tau}\Big(\exp_{\eta(x)}\big(v(\eta(x))\big)\Big)
\end{equation}
for any $\eta\in \Diffxi^s$ and $v\in K^s_{\id}$. 

Since $\eta\in \Diffxi^s$, we know that $\eta\circ\zeta_{\tau} = \zeta_{\tau}\circ\eta$, so that \eqref{commutingisometry}
becomes 
$$ \Exp_{\eta}(V)\circ\zeta_{\tau} = \Exp_{\eta\circ\zeta_{\tau}}(V\circ\zeta_{\tau}) = \zeta_{\tau}\circ \Exp_{\eta}(V)$$
for any $\tau\in\mathbb{R}$. Since $\sigma:=\Exp_{\eta}(V)$ commutes with the flow $\zeta_{\tau}$ of $\Reeb$, differentiating with respect to $\tau$ at $\tau=0$ gives $\sigma_*\Reeb = \Reeb$, and we conclude that $\Exp_{\eta}$ actually maps $\Omega\cap K^s_{\eta}$ \emph{into} $\Diffxi^s$ (which proves that the left side of \eqref{submanifoldisometrycondition} is a subset of the right side). 

To prove the opposite inclusion, let $\sigma$ be any diffeomorphism in $\Exp_{\eta}[\Omega] \cap \Diffxi^s$; we want to show that $\sigma = \Exp_{\eta}(V)$ for some $V\in \Omega \cap K^s_{\eta}$. We already know that $\Exp_{\eta}$ is a diffeomorphism on $\Omega$, so that there is a unique $V\in \Omega$ with $\Exp_{\eta}(V)=\sigma$; thus we just need to show that $V\in K^s_{\eta}$, i.e., that $V=v\circ\eta$ for some $v$ with $[\Reeb, v]=0$. 

Since $\sigma \in \Diffxi^s$, we know that $\sigma\circ \zeta_{\tau}=\zeta_{\tau}\circ\sigma$; hence for each $x\in M$ we have $$
\zeta_{\tau}\big( \exp_{\eta(x)}(V(x)) \big) = \exp_{\eta(\zeta_{\tau}(x))}\big(V(\zeta_{\tau}(x))\big) =
\exp_{\zeta_{\tau}(\eta(x))}\big(V(\zeta_{\tau}(x))\big),
$$
using the fact that $\eta\in \Diffxi^s$ as well. Since isometries preserve geodesics, we also know that 
$$ \zeta_{\tau}\big(\exp_{\eta(x)}(V(x))\big) = \exp_{\zeta_{\tau}(\eta(x))}\big((\zeta_{\tau})_*V(x)\big).$$
Now by our construction of $\Omega$, we know the supremum norm of $V$ is smaller than the injectivity radius of $M$, so that the minimizing geodesic between $\eta(x)$ and $\sigma(x)$ is \emph{unique} for each $x$. Since both $V(\zeta_{\tau}(x))$ and $(\zeta_{\tau})_*V(x)$ are initial velocity vectors for the geodesic at $\zeta_{\tau}(\eta(x))$ which at time one reaches $\zeta_{\tau}(\gamma(x))$, they must be equal for all sufficiently small $\tau$. Since $V=v\circ\eta$, we conclude that $v\big(\zeta_{\tau}(\eta(x))\big) = (\zeta_{\tau})_*v(\eta(x))$ for all $x$, i.e., that $v\big(\zeta_{\tau}(y)\big) = (\zeta_{\tau})_*v(y)$ for all $y\in M$. Now differentiating with respect to $\tau$ at $\tau=0$, we conclude that $[\Reeb, v]=0$ as desired.
\end{proof}

We now consider the set of $H^s$ diffeomorphisms $\Diffximu^s(M)$ preserving both the Reeb field $\Reeb$ and some volume element $\mu$. For this to be interesting, we want the Reeb field itself to preserve the volume element, so we assume that $\Reeb$ is divergence-free. Of course this will be the case if $E$ is Killing. We will prove that $\Diffximu^s(M)$ is a $C^{\infty}$ submanifold of $\Diffxi^s(M)$; this is proved in Smolentsev~\cite{smolentsev_vectorfield} using a different technique.

\begin{theorem}\label{diffximusubmanifold}
Let $M$ be a compact manifold. Suppose $\Reeb$ is a smooth vector field with closed orbits all of the same period, and let $\mu$ be a volume form invariant under the flow of $\Reeb$ (so that $\Reeb$ is divergence-free). Let $\Diffximu^s(M)$ denote the set of $H^s$ diffeomorphisms $\eta$ such that $\eta_*\Reeb=\Reeb$ and $\eta^*\mu = \mu$. Then $\Diffximu^s$ is a smooth submanifold of $\Diffxi^s$. 
\end{theorem}

\begin{proof}
The easiest way to proceed is to use the implicit function theorem. Consider the function $\pullback$ given by $\pullback(\eta) = \eta^*\mu$. If $\eta\in \Diff^s$ then $\eta^*\mu$ is an $H^{s-1}$ $n$-form, given explicitly in coordinates by the Jacobian determinant of $\eta$; since this just involves multiplying the first derivatives of $\eta$ together, and since multiplication of $H^{s-1}$ functions is smooth~\cite{ebinthesis}, we know that $\pullback$ is a $C^{\infty}$ function from $\Diff^s(M)$ to $H^{s-1}(\Lambda^n)$. 

However we want to consider $\pullback$ on the $\Reeb$-invariant diffeomorphisms, which will necessarily map into the $\Reeb$-invariant $n$-forms. Since $\Diffxi^s$ is a smooth submanifold of $\Diff^s$, the restriction of $\pullback$ is smooth on $\Diffxi^s$. If $\zeta_{\tau}$ again denotes the flow of $\Reeb$, our assumptions imply that $\zeta_{\tau}^*\mu=\mu$, so that if $\eta\in \Diffxi^s$ and $\pushedvolume=\eta^*\mu$, we must have 
$$\zeta_{\tau}^*\pushedvolume = (\zeta_{\tau})^*(\eta^*\mu) = \eta^*(\zeta_{\tau})^*\mu = \eta^*\mu = \pushedvolume.$$
We can write any such $H^{s-1}$ volume form $\pushedvolume$ as $\pushedvolume = F\mu$ for some $H^{s-1}$ function $F$ such that $\int_M F\mu = \int_M \mu$. Also since $\mu$ is $\Reeb$-invariant, the form $\pushedvolume$ will be $\Reeb$-invariant if and only if the function $F$ is. Let  
$$H^{s-1}_{\Reeb,1}(\Lambda^n)= \left\{\pushedvolume \in H^{s-1}_E \, \Big\vert \, \int_M \pushedvolume = \int_M \mu\right\},$$
 a space of $\Reeb$-invariant $n$-forms.  We can identify this with $H^{s-1}_{\Reeb,1}(M, \mathbb{R})$, the space of $\Reeb$-invariant functions with mean one. (The fact that any such function is the image of some diffeomorphism $\eta$ follows from Moser's result \cite{Mos}, as generalized to the Sobolev case in \cite{EM}.) One can prove exactly as in Theorem \ref{diffxisubmfd}
that $H^{s-1}_{\Reeb}(M, \mathbb{R})$ is a $C^{\infty}$ submanifold of $H^{s-1}(M, \mathbb{R})$, the space of all $H^{s-1}$ functions on $M$; thus in particular it is a $C^{\infty}$ Hilbert manifold.

We thus have a $C^{\infty}$ function $\pullback$ from the Hilbert manifold $\Diffxi^s$ to the Hilbert manifold $H^{s-1}_{E,1}(\Lambda^n)$, and the inverse image $\pullback^{-1}\{\mu\}$ is the set $\Diffximu^s$; thus to prove $\Diffximu^s$ is a submanifold it is sufficient to show that the derivative of $\pullback$ is surjective in each tangent space. We have $(T\pullback)_{\eta}(v\circ\eta) = \eta^*(\Lie_v\mu)$ for any $\eta\in \Diffxi^s$ and $v\in T_{\id}\Diffxi^s$, and since $\eta^*$ is an isomorphism of $H^{s-1}$ spaces of $n$-forms, it is sufficient to show that $(T\pullback)_{\id}$ is surjective as a map from $T_{\id}\Diffxi^s$ to $T_{\mu}H^{s-1}_{\Reeb}(\Lambda^n)$. Since every $\pushedvolume\in H^{s-1}_{\Reeb,1}(\Lambda^n)$ has the same integral as $\mu$, this tangent space consists of the $\Reeb$-invariant $H^{s-1}$ $n$-forms which integrate to zero over $M$.  This space can be identified with $H^{s-1}_{E,0}(M, \mathbb{R})$, the space of $E$-invariant functions whose integral is zero.

So let $\lambda$ be an $H^{s-1}$ $E$-invariant $n$-form with $\int_M \lambda =0$. Then $\lambda = f\mu$ for some $f\in H^{s-1}_{E,0}((M, \mathbb{R}).$  Choose a Riemannian metric on $M$ such that $\mu$ is its Riemannian volume form, and integrate over the orbits of $\Reeb$ to obtain a metric such that $\Reeb$ is a Killing field; since $\Reeb$ preserves $\mu$, the new metric will still have $\mu$ as its volume form. Using the Hodge decomposition, we can write $f=\diver{v}$, where $v=\grad \Laplacian^{-1}f$ is an $H^s$ vector field.  
A direct calculation shows that $\Lie_v\mu = \lambda$. Since $\Reeb$ is a Killing field, it commutes with the Laplacian, and we have $[\Reeb, v] = \grad \Reeb(\Laplacian^{-1}f) = \grad \Laplacian^{-1}\Reeb(f) = 0$, so that $v\in T_{\id}\Diffxi^s$ as desired. Thus $(T\pullback)_{\id}$ is surjective, and the implicit function theorem implies that $\Diffximu^s$ is a smooth submanifold of $\Diffxi^s$.
\end{proof}

\begin{corollary}\label{diffximudiffmusubmfd}
Under the assumptions of Theorem \ref{diffximusubmanifold}, the space $\Diffximu^s(M)$ is a $C^{\infty}$ submanifold of $\Diffmu^s(M)$, the space of $H^s$ diffeomorphisms preserving the volume form $\mu$.
\end{corollary}

\begin{proof}
Combining Theorem \ref{diffximusubmanifold} and Theorem \ref{diffxisubmfd}, we see that $\Diffximu^s$ is a smooth submanifold of $\Diff^s$. Since $\Diffximu^s$ is a subset of $\Diffmu^s$, and $\Diffmu^s$ is a smooth submanifold of $\Diff^s$ by \cite{EM}, Lemma \ref{3submanifold} gives the result.
\end{proof}

This Corollary will be more useful to us than Theorem \ref{diffximusubmanifold}, since it allows us to obtain a connection on $\Diffximu^s$ directly from the one on $\Diffmu^s$ constructed in \cite{EM} rather than having to construct it  by imitating the proof there.

Now we discuss the quantomorphism subgroup. Ratiu and Schmid \cite{RS} proved that this is a $C^{\infty}$ submanifold of $\Diffxi^s(M)$ (the space of $H^s$ diffeomorphisms preserving the Reeb field $\Reeb$); we will use a modified technique to prove it is also a $C^{\infty}$ submanifold of $\Diffximu^s(M)$. The latter space is more useful geometrically since it has a natural right-invariant $L^2$ Riemannian metric, while the natural  $L^2$ metric on $\Diffxi^s(M)$ is not right-invariant.

The following Lemma is useful in the proof.

\begin{lemma}\label{hodgeReeblemma}
Suppose $\Reeb$ is a Killing field on a compact Riemannian manifold $M$ (without boundary) with closed orbits. Let $\beta$ be a $1$-form with Hodge decomposition 
\begin{equation}\label{hodgebeta}
\beta = d\phi + \omega
\end{equation}
where $\phi$ is a function and $\omega$ is a $1$-form with $\delta \omega=0$. 
If $\beta(\Reeb)=0$ and $\iota_{\Reeb}d\beta=0$, then the same is true for each term: $d\phi(\Reeb)=0$, $\omega(\Reeb)=0$, and $\iota_{\Reeb}d\omega=0$. 
\end{lemma}

\begin{proof}
By Cartan's formula, $\Lie_{\Reeb}\sigma = \iota_{\Reeb}d\sigma + d(\iota_{\Reeb}\sigma)$ for any $1$-form $\sigma$, and hence $\Lie_{\Reeb}$ commutes with $d$ for any vector field $\Reeb$. Typically $\Lie_{\Reeb}$ does not commute with $\delta$; however if $\Reeb$ is Killing it does, because $\delta$ is the adjoint of $d$. Hence for any $k$-forms $\beta$ and $\sigma$ we have 
$ \int_M g(\Lie_{\Reeb}\beta, \sigma) \, d\mu = -\int_M g(\beta, \Lie_{\Reeb}\sigma)\, d\mu$
since $\Lie_{\Reeb}g=0$ for the induced metric on $k$-forms and $\Lie_{\Reeb}\mu=0$ for the Riemannian volume form. Hence $\Lie_{\Reeb}$ is skew-selfadjoint, so that 
$$ \delta \Lie_{\Reeb} = -d^*(\Lie_{\Reeb})^* = -(\Lie_{\Reeb}d)^* = -(d\Lie_{\Reeb})^* = -\Lie_{\Reeb}^* d^* = \Lie_{\Reeb}\delta.$$

By assumption we have $\Lie_{\Reeb}\beta = 0$, and so we have 
$ 0 = d(\Lie_{\Reeb}\phi) + \Lie_{\Reeb}\omega;$
since $\delta\Lie_{\Reeb}\omega = \Lie_{\Reeb}(\delta \omega)=0$, this must be the Hodge decomposition of the zero $1$-form. But the Hodge decomposition of a $1$-form is unique
so $d\Lie_E \phi =0$ and $\Lie_{\Reeb} \omega =0.$ Thus $\Lie_{\Reeb} \phi$ must be a constant. But $\Lie_{\Reeb}\phi = \Reeb(\phi)$, and since $\Reeb$ has closed orbits, the only way $\Reeb(\phi)$ can be constant is if $\Reeb(\phi)=0$. Now applying both sides of \eqref{hodgebeta} to $\Reeb$ we obtain $\omega(\Reeb)=0$. Finally since $\Lie_{\Reeb}\omega=0$ and $\omega(\Reeb)=0$, we must have $\iota_{\Reeb}d\omega=0$.
\end{proof} 

Now we prove $\Diffexcontact^s(M)$ is a smooth submanifold of $\Diffxi^s(M)$. Ratiu and Schmid~\cite{RS} gave a proof of this with a minor error, which we fix using Lemma \ref{hodgeReeblemma}. 

\begin{theorem}\label{quantosubmanifold}
Let $M$ be a contact manifold of dimension $2n+1$ with contact form $\contact$ and regular Reeb field $\Reeb$, and assume the volume form $\mu$ on $M$ is a constant multiple of $\contact \wedge (d\contact)^n$. Then for $s>n+3/2$, the space $\Diffexcontact^s(M)$ of $H^s$ diffeomorphisms preserving $\contact$ is a $C^{\infty}$ submanifold of the space $\Diffxi^s(M)$ of $H^s$ diffeomorphisms with symmetry. 
\end{theorem}

\begin{proof}
We use the implicit function theorem for Hilbert manifolds. By definition the quantomorphism group $\Diffexcontact^s(M)$ is the set of $\eta$ such that $\eta^*\contact = \contact$, so it is tempting to simply use the function $\eta\mapsto \eta^*\contact$ in the implicit function theorem; however there are a couple of subtleties that arise, primarily due to the fact that although $\eta^*\contact$ is only in $H^{s-1}$ for an $H^s$ diffeomorphism $\eta$, we also have $d\eta^*\contact = \eta^*d\contact$ which is in $H^{s-1}$. Thus $\eta^*\contact$ is smoother in some directions than in others, and we need to keep track of this in order to show that $\eta\mapsto \eta^*\contact$ is a submersion.

Let $\pullbackimage^{s-1}(M)$ denote the set of all ordered pairs $(\imcont, d\imcont)$ satisfying the following four conditions: 
\begin{align}
\imcont \in H^{s-1}(\Lambda^1) \quad &\text{and}\quad d\imcont \in H^{s-1}(\Lambda^2); \label{condition1} \\
\imcont(\Reeb) = 1 \quad &\text{and} \quad \iota_{\Reeb}d\imcont = 0. \label{condition2}
\end{align}
These conditions are motivated by the fact that they are satisfied for any $1$-form $\eta^*\contact$ for $\eta\in \Diffximu^s(M)$, as we will show in a moment. First we will show that $\pullbackimage^{s-1}$ is a smooth Hilbert manifold. Consider the space $H^{s-1}(\Lambda^1)\times H^{s-1}(\Lambda^2)$ (the space of ordered pairs of $H^{s-1}$ $1$-forms and $H^{s-1}$ $2$-forms on $M$, not necessarily related to each other). The subset $\{ (\imcont, d\imcont) \, \vert \, \imcont\in H^{s-1}, d\imcont\in H^{s-1}\}$ given by condition \eqref{condition1} defines a closed linear subspace, and since the maps $\imcont \to \imcont(\Reeb)$ and $d\imcont \to \iota_{\Reeb}d\imcont$ are linear and continuous, condition \eqref{condition2} defines a closed affine subspace of that subspace. Thus $\pullbackimage^{s-1}(M)$ is a $C^{\infty}$ affine manifold with tangent space 
$$ T_{\imcont}\pullbackimage^{s-1}(M) = \{ (\imcontvar, d\imcontvar)\in H^{s-1}(\Lambda^1)\times H^{s-1}(\Lambda^2) \, \vert \, \imcontvar(\Reeb)=0, \iota_{\Reeb}d\imcontvar = 0 \}.$$

Now let $F_{\contact}\colon \Diffxi^s(M) \to H^{s-1}(\Lambda^1)\times H^{s-1}(\Lambda^2) $ denote the map $F_{\contact}(\eta) = (\eta^*\contact, \eta^*d\contact)$. In coordinates such that $\contact = \sum_k \contact_k(x) \,dx^k$, the $1$-form of $F_{\contact}(\eta)$ is given by 
$$ (\eta^*\contact)(x) = \sum_{j,k} \contact_k\big(\eta(x)\big) \, \frac{\partial \eta^k}{\partial x^j} \, dx^j;$$
thus if $\eta$ is $H^s$, this is an $H^{s-1}$ $1$-form; similarly $d(\eta^*\contact)=\eta^*d\contact$ is an $H^{s-1}$ $2$-form. Since multiplication and composition with $C^{\infty}$ functions are smooth for $H^{s-1}$ functions with $s-1>\dim{M}/2$~\cite{ebinthesis}, $F_{\contact}$ is a $C^{\infty}$ map. Since any $\eta\in \Diffxi^s(M)$ preserves the Reeb field, it is easy to see that $\phi=\eta^*\contact$ satisfies $\phi(\Reeb)=\contact(\Reeb)=1$ and $\iota_{\Reeb}d\phi = \iota_{\Reeb}d\contact=0$. 
Thus the image of $F_{\contact}$ is contained in the space $\pullbackimage^{s-1}$.  We will show that $F_{\contact}$ is locally surjective onto it.

By definition of the Lie derivative, the derivative of $F_{\contact}$ is the map $(TF_{\contact})_{\eta}(u\circ\eta) = \eta^*\Lie_u\contact$ for any $\eta\in \Diffximu^s(M)$ and any $u\in T_{\id}\Diffximu^s(M)$. Thus  to show that $TF_{\contact}$ is surjective everywhere it is sufficient to show it is surjective at the identity, since $\eta^*$ and composition with $\eta$ are both isomorphisms.

Let $(\beta, d\beta) \in T_{\contact}\pullbackimage^{s-1}(M)$; we want to find $u\in T_{\id}\Diffxi^s(M)$ such that $\Lie_u\contact = \beta$ and therefore $\Lie_u d\contact = d\beta$. In other words we want $u=f\Reeb + v$ where $f\in H^{s}(M,\mathbb{R})$, $v\in H^{s}(TM)$, 
$\Reeb(f)=0,$ $\contact(v)=0$ and $[\Reeb, v]=0$, such that $\beta = \Lie_u\contact = df + \iota_vd\contact$. The fact that we have to differentiate $f$ and not $v$ in this formula is what makes this a bit tricky. Since $\beta$ is in $H^{s-1}$, we can write $\beta = d\phi + \omega$ where $\phi$ is a function in $H^{s}$ and $\omega$ is an $H^{s-1}$ $1$-form with $\delta \omega=0$. However since we know $d\beta$ is also in $H^{s-1}$, and since we have $d\beta = d\omega = (d+\delta)\omega,$ we get $\delta d \beta = \Delta \omega.$   We conclude that $\omega$ is actually in $H^{s}$. Also by Lemma \ref{hodgeReeblemma}, we conclude that $$\Reeb(\phi)=0, \quad \omega(\Reeb)=0, \quad \text{and}\; \iota_{\Reeb}d\omega=0.$$ The first formula implies that we can choose $f=\phi$. The second formula implies that there is a unique vector field $v$ such that $\contact(v)=0$ and $\iota_vd\contact = \omega$, because the map $\gamma$ defined in Proposition \ref{gammaoperatordef} is invertible as a map from the null space of $\contact$ to the annihilator of $\Reeb$ at every point. Since $d\contact$ is smooth, we see that $v$ is also in $H^{s}$ since $\omega$ is.

Finally, the formula $\iota_{\Reeb}d\omega=0$ implies that $[\Reeb, v]=0$; this comes from combining the fact that $\contact([\Reeb, v]) = \Reeb(\contact(v)) - v(\contact(\Reeb)) - d\contact(\Reeb, v) = 0$ with the fact that $\Lie_{[\Reeb, v]}\contact = \Lie_{\Reeb}\Lie_v\contact - \Lie_v\Lie_{\Reeb}\contact = \Lie_{\Reeb}\omega  = 0$ to obtain $\iota_{[\Reeb, v]}d\contact = 0$. Since $[\Reeb, v]$ is in the null space of both $\contact$ and $d\contact$ at every point, it must be zero everywhere. 

We have thus shown that $F_{\contact}$ is a submersion from the smooth manifold $\Diffsxi(M)$ to the 
affine space $\pullbackimage^{s-1}(M)$, and thus $\Diffexcontact^s(M) = F_{\contact}^{-1}(\contact, d\contact)$ is a smooth submanifold of $\Diffsxi(M)$.
\end{proof}

\begin{corollary}\label{quantodiffximu}
Under the same assumptions as Theorem \ref{quantosubmanifold}, the quantomorphism manifold $\Diffexcontact^s(M)$ is a $C^{\infty}$ submanifold of $\Diffximu^s(M)$, the space of volume-preserving maps preserving $\Reeb$.
\end{corollary}

\begin{proof} 
This follows from Lemma \ref{3submanifold}.
\end{proof}

\section{Connections and geodesic equations}

In the last section we were concerned only with the differentiable structure of the diffeomorphism groups. Having established smoothness of all the basic objects, we now investigate the (weak) Riemannian geometry. First we define weak metrics, then we establish the existence of a smooth Levi-Civita connection, and finally we derive the geodesic equation. Smoothness of the connection guarantees that the geodesic equation is a smooth ODE on a Hilbert manifold, and hence we have a smooth Riemannian exponential map (and in particular, local well-posedness of the geodesic equation).

We start with the standard $L^2$ metric on $\Diff^s(M)$, defined for vectors $U,V\in T_{\eta}\Diff^s$ by 
\begin{equation}\label{diffmetric}
\llangle U,V\rrangle_{\eta} = \int_M \langle U(x), V(x)\rangle_{\eta(x)} \, d\mu(x)
\end{equation}
in terms of a Riemannian metric on $M$ and its Riemannian volume form $\mu$. Although this metric is neither right-invariant under the group action nor strong enough to generate the Sobolev topology, it is convenient and natural since its geodesics are given in terms of the pointwise geodesics on $M$. That is, for any $\eta\in \Diff^s$ and $U\in T_{\eta}\Diff^s$, the geodesic on $\Diff^s$ through $\eta$ with initial velocity $U$ is 
\begin{equation}\label{expmapdiff}
\Exp_{\eta}(tU) = x\mapsto \exp_{\eta(x)}(tU(x)).
\end{equation}
The corresponding Levi-Civita connection on $T\Diff^s$ is $C^{\infty}$, as shown in \cite{EM}. For right-invariant vector fields $U$ and $V$ on $\Diff^s(M)$ such that $U_{\eta}=u\circ\eta$ and $V_{\eta}=v\circ\eta$ for vector fields $u$ and $v$ on $M$, the covariant derivative is given by the expected formula $\widetilde{\nabla}_{U_{\eta}}V = (\nabla_uv)\circ\eta$ where $\widetilde{\nabla}$ is the covariant derivative for $\Diff^s$ and $\nabla$ is the covariant derivative of $M$. Given the connection on $\Diff^s$, we will construct connections on submanifolds using orthogonal projections, and if the projection is smooth then the resulting connection will be smooth on the submanifold.

On the smooth submanifold $\Diffmu^s$ of volumorphisms, the restriction of the Riemannian metric \eqref{diffmetric} is right-invariant under the group action. The tangent spaces are given by 
$$ T_{\eta}\Diffmu^s = \{ u\circ \eta \, \vert \, \diver{u} = 0\},$$
and right-invariance implies that the orthogonal projection from the bundle $T\Diff^s\vert_{\Diffmu^s}$ in $T\Diff^s$ to the tangent bundle $T\Diffmu^s$ is given, at any $\eta\in \Diffmu^s$ and $W\in T_{\eta}\Diff^s$ of the form $W=w\circ\eta$ for an $H^s$ vector field $w$ on $M$, by the formula 
$$ P(W)_{\eta} = P(W\circ\eta^{-1})_{\id}\circ\eta, \quad \text{where}\quad P(w)_{\id} = w - \grad \Laplacian^{-1}\diver{w}.$$ 
The fact that $P$ is continuous on $T_{id} \Diff^s$ follows from the Hodge decomposition, but the fact that $P$ is smooth as a function of the base point $\eta$ is more subtle and requires the use of infinite-dimensional fiber bundles~\cite{EM} (we will use a similar technique to prove smoothness of the projection from $T\Diffximu^s\vert_{\Diffexcontact^s}$ to $T\Diffexcontact^s$ in Theorem \ref{projectionsmoothness}). The resulting geodesic equation on $\Diffmu^s$ is the Euler equation 
$$ \frac{D}{dt} \frac{d\eta}{dt} = B\left(\frac{d\eta}{dt}, \frac{d\eta}{dt}\right),$$
written in terms of the covariant derivative of $\Diff^s$ and the second fundamental form $B$; the more familiar form comes from writing $\frac{d\eta}{dt} = u\circ\eta$ and computing the orthogonal projection explicitly to obtain 
$$ \frac{\partial u}{\partial t} + \nabla_uu = -\grad p,$$
the Euler equation of ideal fluid mechanics. Smoothness of the connection implies that this is a smooth ODE on $T\Diffsmu$ and leads to local well-posedness of the Euler equations from the Lagrangian point of view~\cite{EM}.

We shall first study the induced connection on $\Diffximu^s$, the set of $H^s$ diffeomorphisms of $M$ preserving both the volume form and a vector field $\Reeb$ generating a circle action. By Corollary \ref{diffximudiffmusubmfd}, this is a smooth submanifold of $\Diffmu^s$. If the field $\Reeb$ is a Killing field for the given metric, the induced connection makes $\Diffximu^s$ a totally geodesic Riemannian submanifold of $\Diffmu^s$. This implies that if a velocity field solves the Euler equation with initial condition commuting with $\Reeb$, then it will always commute with $\Reeb$. This is well-known in special cases such as axisymmetric Euler flow where $E$ is an infinitesimal rotation about the axis. Of course, $\Diffximu^s$ is not a totally geodesic submanifold of $\Diffxi^s$ for the same reason $\Diffmu^s$ is not totally geodesic in $\Diff^s$; the only fluid flows that are also geodesics in $\Diff^s$ are those for which the pressure function $p$ is a constant function of $x$ for each $t,$ so $\grad p = 0.$

\begin{theorem}\label{reebtotallygeodesic}
Suppose $M$ is a compact Riemannian manifold with Killing field $\Reeb$ with all orbits closed and of the same period. Then in the metric induced by \eqref{diffmetric}, the submanifold $\Diffximu^s(M)$ defined in Theorem \ref{diffximusubmanifold} is a totally geodesic Riemannian submanifold of $\Diffmu^s(M)$. 
\end{theorem}

\begin{proof}
To prove this, we must show that whenever $U$ and $V$ are vector fields on $\Diffximu^s$, the covariant derivative $\nabla_UV$ (computed by extending $U$ and $V$ to vector fields on $\Diffmu^s$ and using the connection on $T\Diffmu^s$) also lies in $T\Diffximu^s$. To make this easier, we convert from covariant derivatives of vector fields to covariant derivatives along curves: for any $\eta\in \Diffximu^s$, let $\gamma$ be a curve such that $\gamma(0)=\eta$ and $\gamma'(0) = U_{\eta}$. Define $W$ to be the vector field along $\gamma$ given by $W(t) = V_{\gamma(t)}$. 

Now $W(t) = w(t) \circ \gamma(t)$ for some time-dependent divergence-free vector field $w$ on $M$, while $U_{\eta} = u\circ\eta$ for some divergence-free vector field $u$ on $M$. By assumption $[\Reeb, u] = 0$ and $[\Reeb, w(t)]=0$ for all $t$. The covariant derivative $\frac{DW}{dt}$ is given on $\Diffmu^s$ by~\cite{misiolek}
$$ \frac{DW}{dt}\big|_{t=0} = \left( \frac{\partial w}{\partial t}\big|_{t=0} + P(\nabla_uw(0))\right)\circ\eta.$$
Since $[\Reeb, w(t)]=0$ for all $t$, we can differentiate to get $[\Reeb, \frac{\partial w}{\partial t}|_{t=0}] = 0$ as well, so we just need to show that $[\Reeb, P(\nabla_uw)]=0$ whenever $\diver{u}=\diver{w}=0$ and $[\Reeb,u]=[\Reeb,w]=0$.

Let $\zeta$ be a diffeomorphism in the flow of $\Reeb$ so $\zeta:M\rightarrow M$ is an isometry.  $[E,u]=[E,v]=0$ implies that $\zeta_*u=u$ and $\zeta_* v = v.$  Also since $\zeta$ is an isometry, it preserves the covariant derivative.  Thus $\zeta_* \nabla_u v = \nabla_u v$.  Furthermore, $\zeta$ preserves the Hodge decomposition of vector fields, so $\zeta_* P \nabla_u v = P \zeta_* \nabla_u v = P \nabla_u v.$  Since this is true for any $\zeta$ in the flow of $\Reeb,$, we find that $[\Reeb, P \nabla_u v]=0.$  Thus $\Diffximu^s(M)$ is totally geodesic in $\Diffmu^s(M).$

\end{proof}


Let us now we use the connection induced on $\Diffximu^s$ to define the connection on $\Diffexcontact^s$ by orthogonal projection. As mentioned above, the hard part is showing that the orthogonal projection depends smoothly on the base point, but once we do this we get local existence of the geodesic equation for free.

\begin{lemma}\label{bundlelemma}
Let $M$ be a contact manifold of dimension $2n+1$ with contact form $\contact$ and Reeb field $\Reeb$ with closed orbits all of the same length, and suppose $s>n+3/2$. There is a unique operator $S_{\contact}$ from $H^{s+1}_{\Reeb}(M, \mathbb{R})$ to $T_{\id}\Diffximu^s(M)$ such that for any $f\in H^{s+1}_{\Reeb}(M, \mathbb{R})$, the vector field $u=S_{\contact}f$ satisfies $\contact(u) = f$ and $\Lie_u\contact = 0$.

Furthermore, if $L_{\contact}$ denotes the operator $L_{\contact}u = \Lie_u\contact$, then $L_{\contact}$ maps $T_{\id}\Diffximu^s$ into $H^{s-1}_{\Reeb}(\Lambda^1)$, the space of $H^{s-1}$ one-forms $\imcont$ such that $\imcont(\Reeb)=0$ and $\iota_{\Reeb}d\imcont = 0$.

Finally, the sequence 
\begin{equation}\label{exactsequence}
0 \to H^{s+1}_{\Reeb}(M, \mathbb{R}) \stackrel{S_{\contact}}{\to} T_{\id}\Diffximu^s(M) \stackrel{L_{\contact}}{\to} H^{s-1}_{\Reeb}(\Lambda^1) \to 0 
\end{equation}
is exact, i.e., the null space of each map is the image of the previous one. In particular the image of $S_{\contact}$ is $T_{\id}\Diffexcontact^s(M)$. 
\end{lemma}

\begin{proof}
We have already shown that $S_{\contact}$ exists in Proposition \ref{quantotangentprop}. 
To prove $S_{\contact}$ has the desired smoothness, we use its expression in Darboux coordinates in \eqref{ScontactDarboux}. 
Since $\Reeb$-invariance implies $\frac{\partial f}{\partial z}=0$, we see $S_{\theta}f$ is in $H^s$ if and only if $f$ is in $H^{s+1}$. 

From above we have $L_{\contact}u = d\big(\contact(u)\big) + \iota_ud\contact$, so that if $u$ is in $H^s$ then $L_{\contact} u$ is only in $H^{s-1}$. The fact that $L_{\contact}$ maps into $H^{s-1}_{\Reeb, \mu}(\Lambda^1)$ follows as in the proof of Theorem \ref{quantosubmanifold}: if $[\Reeb, v]=0$ and $\imcont = \Lie_v\contact$, then $\imcont(\Reeb)=0$ and $\iota_{\Reeb}d\imcont = 0$.
\end{proof}

Denote by $S_{\contact}^{\adj}$ the formal adjoint of the operator $S_{\contact}$ in the $L^2$ metric. To give an example, in Darboux coordinates with the Euclidean metric, it is easy to check that if $w = \sum_k \big(p^k \, \partial_{x^k} + q^k \, \partial_{y^k}\big) + r \, \partial_z$, then
\begin{equation}\label{Scontactadjoint}
S_{\contact}^{\adj}w = (1+n) r + \sum_k \left( \frac{\partial p^k}{\partial y^k} - \frac{\partial q^k}{\partial x^k} + x^k \, \frac{\partial r}{\partial x^k} \right).
\end{equation}
Formally, the orthogonal complement of the image of $S_{\contact}$ is the null space of $S_{\contact}^{\adj}$, and the orthogonal decomposition of an arbitrary $w\in T_{\id}\Diffximu$ is 
\begin{eqnarray} \label{Scontactdecomp}
w & = & P_{\contact}w + (1-P_{\contact})w, \\ \text{where}\quad 
P_{\contact} & = & S_{\contact} \Laplacian_{\contact}^{-1} S_{\contact}^{\adj} \nonumber \\ \text{and}\quad \Laplacian_{\contact} & = & S_{\contact}^{\adj}S_{\contact}. \nonumber
\end{eqnarray}
Again as an example we compute in Darboux coordinates with the Euclidean metric that the contact Laplacian $\Laplacian_{\contact}$ on a (not-necessarily $\Reeb$-invariant) function is
\begin{equation}\label{contactlaplacianDarboux}
\Laplacian_{\contact}f = 2f  - \sum_{k=1}^n \frac{\partial}{\partial x^k}\Big( \big(1+(x^k)^2\big) \, \frac{\partial f}{\partial x^k}\Big) - \sum_{k=1}^n \frac{\partial}{\partial y^k} \frac{\partial f}{\partial y^k}.
\end{equation}
Observe that $\Laplacian_{\contact}$ differentiates twice in all directions except the $z$ (Reeb) direction; this is a general feature. 

Note the similarity to the Hodge decomposition of a vector field into its gradient and divergence-free parts; this allows us to use the same methods to establish smoothness of the orthogonal projection as in \cite{EM}. There are now two issues: showing that the decomposition works for vector fields on $M$ (which requires analyzing the smoothness of $S_{\contact}^{\adj}$ and proving that $\Laplacian_{\contact}$ is invertible), then extending the decomposition from $T_{\id}\Diffximu$ to $T\Diffximu$ (which means by right-invariance that we have to prove that $TR_{\eta}P_{\contact}TR_{\eta}^{-1}$ is smooth in $\eta$).

First we discuss the projection $P_{\contact}$ at the identity. 

\begin{theorem}\label{Scontactadj}
Let $M$ be a Riemannian manifold with contact form $\contact$ and a Reeb field $\Reeb$ which is also a Killing field. 
The formal adjoint $S_{\contact}^{\adj}$ of the operator $S_{\contact}$ defined in Proposition \ref{quantotangentprop} and Lemma \ref{bundlelemma} maps $T_{\id}\Diffximu^s$ into $H^{s-1}_{\Reeb}(M, \mathbb{R})$, and the contact Laplacian $\Laplacian_{\contact} = S_{\contact}^{\adj}S_{\contact}$ is an operator mapping $H^{s+1}_{\Reeb}(M, \mathbb{R})$ bijectively onto $H^{s-1}_{\Reeb}(M, \mathbb{R})$. Hence the projection $P_{\contact}$ defined by \eqref{Scontactdecomp} is a continuous map from $T_{\id}\Diffximu^s(M)$ to itself. 
\end{theorem}

\begin{proof}
We work one tangent space at a time: let $x\in M$; we can then decompose $T_xM = \mathbb{R}\Reeb \oplus \ker \contact$ and $T_x^*M = \mathbb{R}\theta \oplus \ann \Reeb$, where $\ann\Reeb$ is the annihilator of $\Reeb$. Recall that the map $\gamma$ defined in Proposition \ref{gammaoperatordef} is an isomorphism from $\ker \contact$ to $\ann \Reeb$. Denote the inverse map from $\ann \Reeb$ to $\ker \contact$ by $\Gamma$, so that we have $S_{\contact}f = f \Reeb - \Gamma df$ for any $f\in H^{s+1}_{\Reeb}(M, \mathbb{R})$, since $df$ annihilates $\Reeb$. Then the formula 
\begin{align*} 
\llangle S_{\contact}f, w\rrangle &= \int_M \langle f\Reeb - \Gamma df, w\rangle \, d\mu = \int_M f \langle \Reeb, w\rangle - \langle df, \Gamma^*w\rangle \, d\mu \\
&= \int_M f \langle \Reeb, w\rangle - f \delta \Gamma^*w \, d\mu
\end{align*}
implies that $S_{\contact}^{\adj}w = \langle \Reeb, w\rangle - \delta \Gamma^*w$, where $\delta$ is the adjoint of $d$. Hence $S_{\contact}^{\adj}$ is a first-order differential operator which maps $T_{\id}\Diff^s$ into $H^{s-1}(M, \mathbb{R})$. 

To show that $S_{\contact}^{\adj}$ maps $T_{\id}\Diffximu^s$ into $H^{s-1}_{\Reeb}(M, \mathbb{R})$, it is sufficient to show that if $[\Reeb, w]=0$ then $\Reeb(S_{\contact}^{\adj}w)=0$. This will follow from the formula
\begin{equation}\label{ScontactLiecommute}
S_{\contact}\Lie_{\Reeb} = \Lie_{\Reeb}S_{\contact}.
\end{equation}
Since $\Reeb$ is a Killing field, the operators $\Lie_{\Reeb}$ are skew-selfadjoint (on both functions and vector fields) which implies $-\Lie_{\Reeb} S_{\contact}^{\adj} = -S_{\contact}^{\adj}\Lie_{\Reeb}$. The formula \eqref{ScontactLiecommute} is easily demonstrated (whether or not $\Reeb$ is Killing) by writing $S_{\theta}f$ in Darboux coordinates as in \eqref{ScontactDarboux};
clearly in coordinates $S_{\theta}$ commutes with
$\Lie_{\Reeb} = \frac{\partial}{\partial z}$. 

Finally to prove that $\Laplacian_{\contact}: H^{s+1}_{\Reeb}(M, \mathbb{R})\to H^{s-1}_{\Reeb}(M, \mathbb{R}) $ is invertible we recall that $H_E^{s\pm 1}(M)$ is naturally identified with $H^{s\pm 1}(N)$
through $\tilde{f}(\pi(x)):= f(x)$ where $\pi:M\to N$ is the Boothby-Wang fibration.  With this identification we can consider
$ \Laplacian_{\contact}: H^{s+1}(N) \to H^{s-1}(N).$ Furthermore we have
\begin{equation}\label{Laplacianproduct}
\int_M f \Laplacian_{\contact}f \, d\mu = \int_M \lvert S_{\contact}f\rvert^2 \, d\mu = \int_M \lvert f \Reeb - \Gamma df\rvert^2 \, d\mu
\end{equation} 
for any $\Reeb$-invariant function $f$.
Also recall that for every $x\in M$, we have $T_xM = \mathbb{R}\Reeb \oplus \ker \contact$; using the fact that $\Gamma$ is invertible and $\Reeb$ is nowhere zero, we see that for each $x$ there is a constant $c_x>0$ such that \begin{equation}\label{pointwiseelliptic}
\lvert f(x)\Reeb(x) - \Gamma_x df(x) \rvert^2 \ge c_x \big(f(x)^2 + \lvert df(x)\rvert^2\big),
\end{equation} and by compactness of $M$ the numbers $c_x$ are uniformly bounded below.

Since $\Reeb$ is a Killing field on $M$, there is a unique Riemannian metric on $N$ such that the projection $\pi$ is a Riemannian submersion. Hence the operator $\Laplacian_{\contact}$ is an elliptic operator on functions on $N$, and the formulas \eqref{Laplacianproduct} and \eqref{pointwiseelliptic} yield
$$ \int_N f\Laplacian_{\contact}f \, d\nu \ge c \int_N f^2 + \lvert df\rvert^2 \, d\nu,$$
for \emph{any} function $f$ on $N$. (Here $\nu$ is the induced volume form on $N$.) Thus
identifying $H^s_{\Reeb}(M, \mathbb{R})$ with $H^s(N, \mathbb{R})$,
we get a positive definite operator $\Laplacian_{\theta}$ from $H^{s+1}(N, \mathbb{R})$ to $H^{s-1}(N, \mathbb{R})$, which leads to an invertible operator from $H^{s+1}_{\Reeb}(M, \mathbb{R})$ to $H^{s-1}_{\Reeb}(M, \mathbb{R})$. 
\end{proof}

Note that the contact Laplacian is never an elliptic operator on $M$, since it involves two derivatives in all directions except the Reeb direction; see \eqref{contactlaplacianDarboux}. For example with the contact form on $S^3$ described in Example \ref{hopffibration}, we can easily compute that $S_{\contact}f = f E_1 - \tfrac{1}{2} (E_3f) E_2 + \tfrac{1}{2} (E_2f) E_3$, and in the Berger metric we have 
\begin{equation}\label{bergercontactlaplacian}
\Laplacian_{\contact} = \alpha^2 - \tfrac{1}{4} E_2^2 - \tfrac{1}{4} E_3^2.
\end{equation}
On the space of functions with zero Reeb derivative, this causes no problem, but on the space of all functions on $M$, it requires dealing with a \emph{sub-Laplacian} rather than a Laplacian. See \cite{ebinprestoncontacto} for details on these issues. 

Having established that the projection makes sense at the identity, we now want to extend it to a projection from the entire bundle $T\Diffximu^s\vert_{\Diffexcontact^s}$ (the set of all tangent spaces of $\Diffximu^s$ with base points in $\Diffexcontact^s$)
to $T\Diffexcontact^s$. Since the $L^2$ metric on $\Diffximu^s$ is right-invariant, the orthogonal projection is right-invariant as well: that is, for any $\eta \in \Diffexcontact^s$ the projection $(\bundleproj)_{\eta}\colon T_{\eta}\Diffximu^s \to T_{\eta}\Diffexcontact^s$ is given for any $W\in T_{\eta}\Diffximu^s$ by
\begin{equation}\label{subbundleprojection} 
(\bundleproj)_{\eta}(W) = \big(P_{\contact}(W\circ\eta^{-1})\big)\circ\eta,
\end{equation}
which is equivalent to $(\bundleproj)_{\eta} = TR_{\eta} \circ P_{\contact} \circ TR_{\eta^{-1}}$. For this to be smooth on $T\Diffximu^s\vert_{\Diffexcontact^s}$ we need to establish that it is smooth in $\eta$.

Using the formula $P_{\contact} = S_{\contact} \Laplacian_{\contact}^{-1}S_{\contact}^{\adj}$, it is tempting to write 
$$ (\bundleproj)_{\eta} = \big(TR_{\eta} S_{\contact} TR_{\eta^{-1}} \big)\circ \big(TR_{\eta} \Laplacian_{\contact}^{-1} TR_{\eta^{-1}} \big)\circ\big( TR_{\eta} S_{\contact}^{\adj} TR_{\eta^{-1}}\big),$$ 
expecting the three operators to be smooth functions of $\eta$. In fact, for any first-order differential operator $D$ the map $TR_{\eta}\circ D\circ TR_{\eta^{-1}}:H^s \to H^{s-1}$ is smooth in $\eta$ by the Chain Rule; thus $\big(TR_{\eta} S_{\contact} TR_{\eta^{-1}} \big)$ is smooth in $\eta.$ However the middle operator breaks down; for the whole projection to be continuous from $H^s$ to itself, we need the middle operator $F\mapsto \big(\Laplacian_{\contact}^{-1}(F\circ\eta^{-1})\big) \circ \eta$ to map $H^{s-1}$ to $H^{s+1}$; however the composition with $\eta\in \Diff^s$ means that the result is only in $H^s$, which breaks the entire result. Instead, we take an indirect approach following \cite{EM}, Appendix A. We split the projection as 
$$ (\bundleproj)_{\eta} = \big(TR_{\eta} S_{\contact} \Laplacian_{\contact}^{-1} TR_{\eta^{-1}}\big) \circ \big( TR_{\eta} S_{\contact}^{\adj} TR_{\eta^{-1}}\big),$$
and use the fact that the second map is smooth while the first is the inverse of the second on a restricted bundle (thus we obtain smoothness indirectly by the inverse function theorem).

\begin{theorem}\label{projectionsmoothness}
The projection $\bundleproj$ given by \eqref{subbundleprojection} is a $C^{\infty}$ bundle map from the subbundle $T\Diffximu^s\vert_{\Diffexcontact^s}$ to $T\Diffexcontact^s$.
\end{theorem}

\begin{proof}
Since $S_{\contact}^{\adj}$ is a first-order differential operator, the composition $\big(\overline{S_{\contact}^{\adj}}\big)_{\eta} = TR_{\eta}\circ S_{\contact}^{\adj}\circ TR_{\eta^{-1}}$ is given on $W\in T_{\eta}\Diffximu^s$ in a chart (using the chain rule) as a product of the inverse matrix of $D\eta$ and first derivatives of $W$. Since $\eta$ is a diffeomorphism, the matrix $D\eta$ is always invertible, and since $\eta$ is volume-preserving, $D\eta$ has determinant one. Thus the map from $D\eta$ to its inverse is smooth, since by Cramer's rule it involves only determinants of minors, and products of $H^{s-1}$ functions are still in $H^{s-1}$ because $s>\dim{M}/2+1$ \cite{ebinthesis}. We thus find $\overline{S_{\contact}^{\adj}}$ is smooth in any chart; hence it is smooth as a bundle map from $T\Diffximu^s(M)$ to $\Diffximu^s(M) \times H^{s-1}(M, \mathbb{R})$.

Now consider $T\Diffexcontact^s$ as a subset of $T\Diffximu^s\vert_{\Diffexcontact^s}$. This is a subbundle since it is the null space of the bundle map $\overline{L_{\contact}}$ defined by $(\overline{L_{\contact}})_{\eta} = TR_{\eta} \circ L_{\contact} \circ TR_{\eta^{-1}}$ as in Lemma \ref{bundlelemma}, and $\overline{L_{\contact}}$ is smooth since as before it comes from a first-order differential operator.

Since $T\Diffexcontact^s$ is a smooth subbundle, the restriction of $\overline{S_{\contact}^{\adj}}$ to $T\Diffexcontact^s$ is still smooth. On this subbundle it is a bundle isomorphism onto its image $\Diffexcontact^s\times H^{s-1}_{\Reeb}(M, \mathbb{R})$, since $T\Diffexcontact^s$ is the image under the bundle map $\overline{S_{\contact}}$ of the bundle $\Diffexcontact^s\times H^{s+1}_{\Reeb}(M, \mathbb{R})$. Its inverse is given at each $\eta$ by 
$\big(\overline{S_{\contact}^{\adj}})^{-1}_{\eta} = TR_{\eta} (S_{\contact}\Laplacian_{\contact}^{-1}) TR_{\eta^{-1}}$. But the inverse of a smooth isomorphism is also smooth by the inverse function theorem. Thus the map $(\overline{S_{\contact}^{\adj}})^{-1}$ is a smooth bundle map from $\Diffexcontact^s\times H^{s-1}_{\Reeb}(M, \mathbb{R})$ to $T\Diffexcontact^s$. 
Since $T\Diffexcontact^s$ is a smooth subbundle, this bundle map is also smooth when considered as a map from $\Diffexcontact^s\times H^{s-1}_{\Reeb}(M, \mathbb{R})$ into $T\Diffximu^s\vert_{\Diffexcontact^s}$. Hence the composition $\bundleproj$ is also smooth as a bundle map. 
\end{proof}

Now the orthogonal projection defines the Levi-Civita covariant derivative on the submanifold $\Diffexcontact^s$ of $\Diffximu^s$, and Theorem \ref{projectionsmoothness} implies that the induced Levi-Civita connection is $C^{\infty}$. This will allow us to discuss the geodesic equation in the next section.

\section{Local and global existence for the geodesic equation}

Since $\Diffexcontact^s$ is a Riemannian submanifold of $\Diffximu^s$, the covariant derivative on $\Diffexcontact^s$ comes from the orthogonal projection of the covariant derivative on $\Diffximu^s$, which we know is smooth. Thus the geodesic equation on $\Diffexcontact^s$ is a $C^{\infty}$ ordinary differential equation, which is guaranteed to have local-in-time solutions by the Picard method.

\begin{theorem}\label{localgeodesic}
Suppose $M$ is a Riemannian contact manifold, with contact form $\contact$ and regular Reeb field $\Reeb$, such that $\Reeb$ is Killing. Then for $s> \dim(M)/2 +1$ , the geodesic equation $ \big(\overline{P_{\contact}}\big)_{\eta}\left(\frac{D}{dt} \frac{d\eta}{dt}\right) = 0$ on $\Diffexcontact^s$ is equivalent to
\begin{equation}\label{geodesic}
\frac{\partial \eta}{\partial t} = S_{\contact}f\circ\eta, \qquad \frac{\partial \Laplacian_{\contact}f}{\partial t} + \{f, \Laplacian_{\contact}f\} = 0,
\end{equation}
where the contact bracket is defined for functions $f$ and $g$ in $H^s_{\Reeb}(M,\mathbb{R})$ as in \eqref{contactbracket} by $\{f,g\} = S_{\contact}f(g)$. 

The geodesic equation \eqref{geodesic} is a smooth ODE on $\Diffexcontact$, and hence there is a smooth exponential map $\exp_{\id}\colon \Omega\subset T_{\id}\Diffexcontact^s\to \Diffexcontact^s$ such that $\exp_{\id}(v)$ is the geodesic $\eta(1)$ where $\Omega$ is a neighborhood of zero in $T_{\id}\Diffexcontact^s,$ $\eta(0)=\id$ and $\eta'(0)=v$. 
\end{theorem}

\begin{proof}
Let $\frac{D}{dt}$ denote the covariant derivative on $\Diffximu^s$; then a geodesic in $\Diffximu^s$ satisfies $\frac{D}{dt} \frac{d\eta}{dt} = 0$. Since $\Diffexcontact^s$ is a smooth submanifold, its geodesic equation is $ \big(\overline{P_{\contact}}\big)_{\eta}\left(\frac{D}{dt} \frac{d\eta}{dt}\right) = 0$.
Having established smoothness of $\big(\overline{P_{\contact}}\big)_{\eta}$ in Theorem \ref{projectionsmoothness}, we know this geodesic equation is $C^{\infty}$ on $\Diffexcontact^s$. 
Since the geodesic equation is smooth, existence of an exponential map follows from the usual theory of ODEs in Banach manifolds.

To compute the geodesic equation explicitly, we write $\frac{d\eta}{dt} = u\circ\eta$ for a time-dependent vector field $u$ on $M$, which is in $T_{\id}\Diffexcontact^s$. We can write $u=S_{\contact}f$ for some time-dependent function $f$ in $H_E^{s+1}(M)$. In the ambient space $\Diff^s(M)$, we know the covariant derivative is given by $ \frac{\widetilde{D}}{dt} \frac{d\eta}{dt} = (u_t + \nabla_uu)\circ\eta$, and the covariant derivative in $\Diffexcontact^s$ is just the orthogonal projection onto $T\Diffexcontact^s$. To compute this, we let $v=S_{\contact}g$ be an arbitrary vector in $T_{\id}\Diffexcontact^s$, and define a right-invariant vector field by $V_{\eta}=v\circ\eta$. Then
\begin{align*} 
0 &= \Llangle \frac{\widetilde{D}}{dt} \frac{d\eta}{dt}, V_{\eta}\Rrangle_{\eta} = \llangle u_t + \nabla_uu, v\rrangle_{\id} \\
&= \int_M \big(\langle u_t, v\rangle - \langle u,v\rangle \diver{u} - \langle u, [u,v]\rangle + \tfrac{1}{2} \lvert u\rvert^2 \diver{v} \big) \, d\mu.
\end{align*}
Now $\diver{u}=\diver{v}=0$ since $\Diffexcontact\subset \Diffximu$. 
Also $[u,v] = [S_{\contact}f, S_{\contact}g] = S_{\contact}\{f,g\}$ where the contact bracket is given as in \eqref{contactbracket}. Thus we find 
$$ 0 = \int_M \big(g \Laplacian_{\contact}f_t - \{f,g\} \Laplacian_{\contact}f \big) \, d\mu = 
\int_M g\big( \Laplacian_{\contact}f_t + \{f, \Laplacian_{\contact}f\}\big) \, d\mu.$$ 
where the last equation comes from the standard property of Poisson bracket which was mentioned after the proof of Proposition \ref{bracketisomorphism}.

Since this is true for any function $g\in H^{s+1}_{\Reeb}(M)$, we must have \eqref{geodesic}.\footnote{Note that the decomposition \eqref{geodesic} is valid is $H^{s-1}$, but not in $H^s$. The two equations are not actually ODEs on $\Diffexcontact^s$, since left-translation is not differentiable in the $H^s$ spaces. The second formula in \eqref{geodesic} is an expression of the general formula $ \frac{du}{dt} + \ad_u^*u = 0$ for the Euler-Arnold equation on a Lie group with right-invariant metric. This was originally derived by
Arnold~\cite{arnold} in the general case; see for example \cite{arnoldkhesin} or \cite{MP} for more recent references.}
\end{proof}

\begin{exmp}\label{geostrophicsphere}
Recall that for the contact structure on $S^3$ as in Example \ref{hopffibration}, the contact Laplacian is given by \eqref{bergercontactlaplacian}: $\Laplacian_{\theta} = \alpha^2 - \tfrac{1}{4} (E_2^2 + E_3^2)$. As mentioned, this operator is more naturally understood in terms of the Boothby-Wang quotient, which in this case is $N=S^2$. We can then compute that the operator $\Laplacian_{\theta}$ reduces to $\alpha^2 - \Laplacian$ on $S^2$ in terms of the usual Laplacian on $S^2$ (the factor $\frac{1}{4}$ cancels, since the metric on $S^2$ which makes the Hopf fibration a Riemannian submersion is $\frac{1}{4}$ of the usual one). Hence the geodesic equation \eqref{geodesic} reduces to 
\begin{equation}\label{spheregeostrophic} 
\partial_t(\Laplacian f - \alpha^2 f) + \{f, \Laplacian f\} = 0
\end{equation}
on $S^2$ in terms of the Poisson bracket (using the fact mentioned after Proposition \ref{bracketisomorphism} that the contact bracket on $\Reeb$-invariant functions is equivalent to the symplectic bracket on the quotient). 

Equation \eqref{spheregeostrophic} is the quasigeostrophic equation on $S^2$ ``in the $f$-plane approximation,'' where the parameter $\alpha^2$ is the rotational Froude number (or $1/\alpha$ is the Rossby number) and $f$ is the stream function; see Holm-Zeitlin~\cite{holmzeitlin} and Zeitlin-Pasmanter~\cite{zeitlinpasmanter}. (More precisely the $f$-plane approximation is valid only when the configuration manifold is $\mathbb{R}^2$; this is a simplifiication of the true quasigeostrophic equation.) In our context the Lagrangian motion is given in $S^3$ by $\frac{\partial \eta}{\partial t} = S_{\theta}f \circ \eta$, so that $\eta \in \Diffexcontact(S^3)$. As we will see in a moment, the quantomorphism group is a circle bundle over the group of Hamiltonian symplectomorphisms, and its Lagrangian flow projects to the usual Lagrangian flow on $S^2$. We thus obtain a natural geodesic interpretation of the $f$-plane quasigeostrophic equation which is a bit different from the one given in \cite{zeitlinpasmanter}. 

This works more generally; given any two-dimensional manifold $N$ with symplectic form $\omega$, we can find a three-dimensional contact manifold $M$ such that $N$ is the Boothby-Wang quotient, and there is a unique choice of Riemannian metric on $M$ such that the Reeb field $\Reeb$ is orthogonal to $\ker \contact$ and $\lvert \Reeb\rvert^2$ is a given constant Froude number; the geodesic equation on $\Diffexcontact(M)$ will then be the ``$f$-plane'' quasigeostrophic equation. Later (see Remark \ref{realquasigeostrophic}) we wills discuss the actual quasigeostrophic equation on the $2$-sphere.
\end{exmp}

From the theory of ODEs on Banach spaces (see \cite{lang}, Chapter 4, Proposition 2.5), we know that the only way a solution can blow up is if the norm blows up. Hence we can analyze global existence by establishing estimates for the Sobolev norms of $u=S_{\theta}f$ in $T_{\id}\Diffexcontact^s(M)$. For this purpose it is convenient to view \eqref{geodesic} as an equation not on the contact manifold $M$ but on its Boothby-Wang symplectic quotient $N$. Hence we consider Hamiltonian diffeomorphisms on $N$.

\begin{defn} We say that a diffeomorphism $\zeta$ is in $\Diffham^s(N)$ if there exists a curve $\zeta(t)$ from $\id$ to
$\zeta$ such that $\dot{\zeta}(t) = u(t) \circ \zeta(t)$ where $u(t) = \omega^{\sharp}(df(t))$ for some $f(t)\in H^{s+1}(N)$ for every $t$. (Recall that $\omega^{\sharp}$ is the inverse of the operator $\omega^{\flat}\colon= v\mapsto \iota_v\omega$.)
\end{defn}

There is a natural fibration $\Pi$ from quantomorphisms $\Diffexcontact^s(M)$ on $M$ to Hamiltonian diffeomorphisms $\Diffham^s(N)$ on $N$, first constructed by Ratiu and Schmid~\cite{RS}. 
The map 
is constructed as follows: if $\eta \in \Diffexcontact^s(M),$ then $\eta$ preserves orbits of $\Reeb$ by Proposition \ref{reebpreserving}. Hence one gets $\zeta \in \Diff^s(N)$ by letting $\zeta(\tilde{x}) = \pi(\eta(x))$ where $\tilde{x} = \pi(x)$.  Since $\eta \in \Diffexcontact^s(M)$, there exists a curve $\eta(t)$ from $\id$ to $\eta$ in $\Diffexcontact^s(M)$ such that $\dot{\eta} = v \circ \eta$ where  
$v = S_{\theta} f$ for some curve $f \in H^{s+1}_{\Reeb}(M).$ Thus we can set $u=\pi_* S_{\theta} f = \pi_*(fE-\Gamma df) = 
-\pi_*(\Gamma df)$, and the flow $\zeta$ of the time-dependent $u$ will be a curve in $\Diff^s(N)$. But for any vector $z\in TM$, we have
$\omega(\pi_* \Gamma df, \pi_* z) = d\contact(\Gamma df, z)=\gamma \Gamma df(z)=df(z).$  Therefore $\omega^{\flat}(\pi_* \Gamma df) = d \tilde{f}$, where $f \in H^{s+1}_E(M)$ is identified with $\tilde{f} \in H^{s+1}(N)$.  Hence $\omega^{\flat}(\pi_*(u)) = -d\tilde{f}.$  It follows that $\zeta \in \Diffham^s(N)$.

\begin{theorem}\label{boothbywang}
Let $M$ be a Riemannian manifold of dimension $2n+1$ with a regular contact form $\contact$ for which the Reeb field $\Reeb$ is a Killing field with orbits of constant length $1$. Suppose that the Riemannian volume form $\mu$ on $M$ is a constant multiple of $\contact \wedge (d\contact)^n$.
Define $\varphi = \Laplacian_{\contact}(1) = S_{\contact}^{\adj}(\Reeb)$; then $\varphi$ is $\Reeb$-invariant and descends to a function $\tilde{\varphi}\colon N\to\mathbb{R}$ such that $\tilde{\varphi}\circ\pi=\varphi$.
Let $N$ be the Boothby-Wang quotient with smooth projection $\pi\colon M\to N$, and $\omega$ the symplectic form on $N$ satisfying $\pi^*\omega = d\contact$. Then there is a unique Riemannian metric on $N$ such that the projection $\pi$ is a Riemannian submersion, and the induced volume form $\nu$ on $N$ satisfies $\pi^*\nu = \iota_{\Reeb}\mu$ and is a constant multiple of $\omega^n$. 

Define a right-invariant Riemannian metric on $\Diffham^s(N)$ at the identity by the formula 
\begin{equation}\label{hamiltonianmetric}
\llangle u, v\rrangle_N = \int_N \tilde{f}\Laplacian_{\contact} \tilde{g} \, d\nu - \frac{\left(\int_N \tilde{\varphi} \tilde{f}\, d\nu\right)\left(\int_N \tilde{\varphi} \tilde{g}\, d\nu\right)}{\int_N \tilde{\varphi} \, d\nu},
\end{equation}
for vector fields $u,v\in T_{\id}\Diffham^s(N)$, 
where $u=\omega^{\sharp}d\tilde{f} $ and $v=\omega^{\sharp}d\tilde{g}$, and
$\Laplacian_{\contact}$ is identified with an operator on $N$ as in the proof of Theorem \ref{Scontactadj}. 
Then the metric is well-defined on $\Diffham^s(N)$, and the map $\Pi$ is also a Riemannian submersion.
\end{theorem}

\begin{proof}
Near any point of $M$ we can choose a Darboux coordinate chart $(x^1, \cdots, x^n, y^1, \cdots, y^n, z)$ such that $\theta = dz + \sum_{k=1}^n x^k \, dy^k$; then the corresponding chart on $N$ is $(u^1, \cdots, u^n, v^1, \cdots, v^n)$, and the projection in coordinates is $$\pi(x^1,\cdots, x^n, y^1, \cdots, y^n, z) = (x^1, \cdots, x^n, y^1, \cdots, y^n),$$ i.e., just ``forgetting about $z$.'' In such coordinates the Reeb field is $\Reeb = \frac{\partial}{\partial z}$, and since the Reeb field is Killing, the components of the Riemannian metric on $M$ are independent of $z$. For each $k$ the horizontal vectors are $\frac{\partial}{\partial x^k} + p^k \, \frac{\partial}{\partial z}$ and $\frac{\partial}{\partial y^k} + q^k \, \frac{\partial}{\partial z}$ for some functions $p^k$ and $q^k$ (chosen to make the vectors orthogonal to $\Reeb$). 
Then if $\pi:M \to N$ is a Riemannian submersion, the metric on $N$ satisfies 
$$\langle \partial_{u^j}, \partial_{u^k} \rangle_N = \langle \partial_{x^j} + p^j \, \partial_z, \partial_{x^k}+p^k\,\partial_z\rangle_M,$$
$$\langle \partial_{u^j}, \partial_{v^k} \rangle_N = \langle \partial_{x^j} + p^j \, \partial_z, \partial_{y^k}+q^k\,\partial_z\rangle_M,$$
$$\langle \partial_{v^j}, \partial_{v^k} \rangle_N = \langle \partial_{y^j} + q^j \, \partial_z, \partial_{y^k}+q^k\,\partial_z\rangle_M.$$
The fact that the right side is independent of $z$ implies that the left side is well-defined on $N$.

Because both Riemannian metrics are right-invariant, it is sufficient to check the submersion condition at the identity. 
The construction of $\Pi$ shows that the null-space of 
$\Pi_*\colon T_{\id}\Diffexcontact^s(M)\to T_{\id}\Diffham^s(N)$ 
consists of vector fields $S_\theta f$ such that $df=0.$  But for any such $f,$
$S_\theta f = fE$ with $f$ constant, so the null-space is one-dimensional and spanned by $E$.
Hence a vector $S_{\contact}f\in T_{\id}\Diffexcontact$ is horizontal if and only if $\llangle S_{\contact}f, \Reeb\rrangle = \int_M f \Laplacian_{\contact}(1) \, d\mu = \int_M f \varphi \, d\mu =0$. Now given any $u\in T_{\id}\Diffham^s(N)$ there is a unique $\tilde{f}\colon N\to\mathbb{R}$ with lift $f:M\to \mathbb{R}$ satisfying $\int_M f \varphi \, d\mu=0$ such that $u=\omega^{\sharp} d\tilde{f}$; then $\Pi_*(S_{\contact}f) = u$, and so $\Pi$ is a Riemannian submersion if and only if $\llangle u, u\rrangle_N = \llangle S_{\contact}f, S_{\contact}f\rrangle_M$ for any such $u$. Since the orbits of $\Reeb$ are assumed to have length $1$, for any Reeb-invariant function $\psi$ on $M$ we have $\int_M \psi \, d\mu = \int_N \tilde{\psi} \, d\nu$ for the induced function $\tilde{\psi}$ on $N$. 

To prove that \eqref{hamiltonianmetric} is well-defined, we consider any vectors $u$ and $v$ in $T_{\id}\Diffham^s(N)$. Then $u=\omega^{\sharp}d\tilde{f}$ and $v=\omega^{\sharp}d\tilde{g}$ for some functions $\tilde{f}$ and $\tilde{g}$ on $N$, which are defined only up to a constant. Now let $f$ and $g$ be the lifts of $\tilde{f}$ and $\tilde{g}$ to $M$, and write $\overline{f}=f+a$ and $\overline{g}=g+b$, where we will demand that $S_{\contact}\overline{f}$ and $S_{\contact}\overline{g}$ are horizontal. Then we must have $\int_M \overline{f} \varphi \, d\mu = \int_M \overline{g}\varphi \, d\mu = 0$, which implies that 
$$ \textstyle a = -(\int_M f\varphi \, d\mu)/(\int_M \varphi \, d\mu) \qquad \text{and}\qquad b = -(\int_M g\varphi \, d\mu)/(\int_M \varphi \, d\mu).$$ 
We compute that 
\begin{align*}
\textstyle \llangle S_{\contact}\overline{f}, S_{\contact}\overline{g}\rrangle &= \textstyle \int_M f\Laplacian_{\contact} g\, d\mu 
+ a \int_M g S_{\contact}(1) \, d\mu\\
&\qquad\qquad  \textstyle + b\int_M f S_{\contact}(1) \, d\mu + ab \int_M S_{\contact}1 \, d\mu \\
&= \textstyle  \int_M f\Laplacian_{\contact}g \, d\mu - (\int_M f\varphi \,d\mu) (\int_M g\varphi\, d\mu)/(\int_M \varphi \, d\mu) \\
&= \textstyle \int_N \tilde{f}\Laplacian_{\contact}\tilde{g} \, d\nu - (\int_N \tilde{f}\tilde{\varphi}\,d\nu)(\int_N \tilde{g} \tilde{\varphi}\,d\nu)/(\int_N \tilde{\varphi}\, d\nu).
\end{align*}
Hence although $\tilde{f}$ and $\tilde{g}$ are only defined up to constants, the metric \eqref{hamiltonianmetric} gives the same result for any choice of the constant and is thus well-defined as an inner product on the Hamiltonian vector fields $u$ and $v$ on $\Diffham^s(N)$. 
\end{proof}

The metric \eqref{hamiltonianmetric} never reduces to the $L^2$ metric studied in \cite{ebinsymplectic}, since that metric depends only on derivatives of the stream function. \eqref{hamiltonianmetric} is essentially a hybrid between the metric in \cite{ebinsymplectic} and the $L^2$ Hofer metric on stream functions (which only generates a pseudometric on the space of Hamiltonian diffeomorphisms; see \cite{ElPol}.) 
In the situation described in Example \ref{geostrophicsphere}, 
the induced metric \eqref{hamiltonianmetric} on symplectic diffeomorphisms of $S^2$ is exactly the $H^1$ metric on stream functions.
We must be careful, however: since the tangent space of $\Diffham(N)$ consists of skew gradients of functions, it can only be identified with functions defined up to some constant. Hence the metric \eqref{hamiltonianmetric} is necessarily degenerate in one direction. Nonetheless we can still write a geodesic equation as an Euler-Arnold equation on the corresponding homogeneous space (see Khesin-Misiolek~\cite{KM}).

\begin{proposition}\label{hamiltonianquotientgeodesic}
On the space $\Diffham^s(N)$ with right-invariant degenerate Riemannian metric given by \eqref{hamiltonianmetric}, the Euler-Arnold equation on the homogeneous quotient space is given by 
\begin{equation}\label{reducedgeodesic} 
\partial_t\Laplacian_{\contact}\tilde{f} + \{\tilde{f}, \Laplacian_{\contact}\tilde{f}\} = 0,
\end{equation}
where $\{\cdot, \cdot\}$ is the symplectic Poisson bracket and $\tilde{f}$ is determined only up to addition by a constant.
\end{proposition}

\begin{proof}
We just need to compute the Euler-Arnold equation $u_t + \ad_u^{\star}u = 0$ for $u=\omega^{\sharp}d\tilde{f}$ in the metric \eqref{hamiltonianmetric}. Using $\ad_uv = -\omega^{\sharp}\{\tilde{f},\tilde{g}\}$ for $u=\omega^{\sharp}d\tilde{f}$ and $v=\omega^{\sharp}d\tilde{g}$, we can easily compute that the Euler-Arnold equation is
$$ \partial_t\Laplacian_{\contact}\tilde{f} - \tfrac{\int_N \tilde{\varphi} \partial_t \tilde{f}\,d\nu}{\int_N \tilde{\varphi}\,d\nu} + \{\tilde{f}, \Laplacian_{\contact}\tilde{f}\} - \tfrac{\int_N \tilde{\varphi}\tilde{f}\,d\nu}{\int_N \tilde{\varphi}\,d\nu} \{\tilde{f}, \tilde{\varphi}\} = 0.$$
Choosing a representative $\tilde{f}$ such that $\int_N \tilde{f}\tilde{\varphi}\,d\nu=0$, we obtain \eqref{reducedgeodesic}.
\end{proof}

The geodesic equation \eqref{geodesic} is equivalent to the equation \eqref{reducedgeodesic}
for a time-dependent function $\tilde{f}\colon N\to\mathbb{R}$, where $\Laplacian_{\contact}$ is the inherited operator defined as in Theorem \ref{boothbywang}. Global existence for the quantomorphism geodesic equation then reduces to showing that this equation has global solutions for sufficiently smooth initial data: since we want $S_{\contact}f$ to initially be in $H^s$, 
we assume that $f\in H^{s+1}(M,\mathbb{R})$. Global existence follows in much the same way as in \cite{ebinsymplectic}: we first bound the Sobolev norms of $f$ in terms of the $C^1$ norm, then establish an absolute maximum for the $C^1$ norm. In that paper the proof was very brief; here we give more detail. 

\begin{theorem}\label{globalgeodesic}
If the initial condition $f$ is in $H^{s+1}_{\Reeb}(M, \mathbb{R})$ for $s>\dim{N}/2+1$ (i.e., if $\tilde{f}$ is in $H^{s+1}(N,\mathbb{R}$)), then solutions of equation \eqref{reducedgeodesic} exist for all time. Hence the exponential map given by Theorem \ref{localgeodesic} is defined on the entire tangent space $T_{\id}\Diffsexcontact(M)$. Thus by right invariance it is defined on all of $T\Diffexcontact^s(M)$. 
\end{theorem}

\begin{proof}

In order to prove this theorem we recall that the space $H^{s \pm 1}_E (M, \mathbb{R})$  is naturally identified with $H^{s \pm 1}(N, \mathbb{R}).$  With this identification we get
$$ \Laplacian_{\contact}: H^{s + 1}(N, \mathbb{R})\rightarrow H^{s - 1}(N, \mathbb{R}) $$

  We shall deal with three time dependent quantities whose estimates will be interdependent.  We have the stream function $f(t) \in H^{s+1}_{\Reeb}(M, \mathbb{R}) \cong H^{s+1}(N, \mathbb{R})$, the velocity field $u(t) = S_{\theta} f(t) \in  T_{id} \Diffexcontact^s(M) $ and the Lagrangian flow $\eta(t)$ in $\Diffexcontact^s(M)$.  Let $(-T_b, T_e)$ be the maximum time interval on which $f$ (and therefore $u$ and $\eta$) exists.  If $T_e$ were finite, then we would have $\lim_{t\nearrow T_e} \|f(t)\|_{s+1} = \infty$, because  if $\|f(t)\|_{s+1}$ remained bounded then $\|u(t)\|_s $ would be bounded and then $T_e$ would not be maximal (see \cite{hartman}).  Hence we need only show that $\|f(t)\|_{s+1} $ is bounded on any finite time interval.  
  
 To do this we first note that
$$ \frac{d}{dt}( \Laplacian_{\contact} f(t) \circ \eta(t)) =
(\partial_t \Laplacian_{\contact} f + S_{\theta}f ( \Laplacian_{\contact} f )) \circ \eta = 0  $$
by \eqref{geodesic}.
Therefore $\Laplacian_{\contact} f(t) \circ \eta(t) = \Laplacian_{\contact} f(0)$, so in particular the $C^0$-norm of $\Laplacian_{\contact} f (t)$ is constant in time.  Using this we will find a bound for $df$, or equivalently for $S_{\contact} f.$

Recall from the proof of Theorem \ref{Scontactadj} that the contact Laplacian $\Laplacian_{\contact}\colon H^{s+1}(N,\mathbb{R}) \to H^{s-1}(N,\mathbb{R})$ is an isomorphism.
From the theory of elliptic operators (\cite{taylor}, Chapter 7, Proposition 2.2 for $N = \mathbb{R}^n$ and Chapter 7, Section 10 for $N$ any compact manifold) we find that its inverse is a pseudodifferential operator whose Schwartz kernel $k: N \times N \rightarrow \mathbb{R} $ is smooth off of the diagonal and obeys the estimate
\begin{equation} \label{pseudo-est} |\partial_x^{\beta} k(x,y)| \leq K \rho(x,y)^{-2n+2-|\beta|} \quad {\rm for} -2n+2-|\beta| < 0 \end{equation}
where $\partial_x $ means a partial derivative with respect to the first variables of $k$ and $\rho(x,y) $ is the distance from $x$ to $y$ defined by the Riemannian metric on $N$.\footnote{In the inequalities that follow $K$ will always denote some positive constant, but it may different in different inequalities.} 
From this if $n>1$, we can estimate
$$ f(x) = \int_N  k(x,y) \Laplacian_{\contact} f(y) dy $$ and
\begin{equation}
\label{df} df(x) = \int_N d_x k(x,y) \Laplacian_{\contact} f(y) dy
\end{equation}
where $d_x$ is the differential with respect to the $x$-variables and $dy$ indicates integration with respect to $y$ using the Riemannian volume element of $N.$\footnote{In the sequel we will sometimes write $dk$ for $d_x k.$}
We have $|k(x,y)| \leq K \rho(x,y) ^{-2n+2}$ and $|d_x k(x,y)| \leq K \rho(x,y) ^{-2n+1},$ so these integrals are  bounded by a constant times $\| \Laplacian_{\contact} f(t)\|_{C^0}$ which, as we have seen, is independent of $t$. With this and formula \eqref{ScontactDarboux} we see that $u=S_{\theta} f$ is bounded uniformly in time as well.
If $n=1$ the estimate for $df$ is still valid, but we must use a different method to estimate $f$.  First we note that if $|df|<K$ then 
\begin{equation}\label{minmaxKR}
\max{f}-\min{f}<KR
\end{equation}
where $R$ is the diameter of $N$.  Also we see from \eqref{reducedgeodesic} that 
$\int_N \Laplacian_{\contact} f$ is constant in time.
Furthermore
$$ \int_N \Laplacian_{\contact} f = \int_N f\Laplacian_{\contact}1=
\int_N f \varphi$$ and $\int_N \varphi = \int_M \lvert E \rvert^2 >0$. Also $\int_N f\varphi \geq \int_N (\min{f}) \varphi = \min{f} \int_N \varphi.$ Thus $\min f$ is bounded so by \eqref{minmaxKR} $f$ is also bounded, and of course all the bounds are uniform in time.

We proceed to seek a time-uniform Lipschitz bound for $df$, but in fact we will be able to find only a quasi-Lipchitz bound, as we shall now explain.  Since $N$ is compact, we can find a positive $\delta$ such that each point of $N$ has a normal coordinate neighbourhood ball of radius at least $\delta.$  Then for $y \in N$ with $\rho(x,y) < \delta$, we have a unique minimal geodesic $\mingeodesic$ parametrized so that $\mingeodesic(0) = x$ and $\mingeodesic (1) = y$. We let  $ b = |\mingeodesic'(\tau)|,$ so that $b = \rho(x,y)$. Also we parallel translate $ df(x)$ along $\mingeodesic$ to get some $df'(y) \in T^* _y N.$ Then we shall estimate
$|df(y) - df'(y)|$ where $|\; \: |$ is the norm on $T^* _y N.$

To do this we use the formula \eqref{df} for $df$, and we parallel translate each $d_x k(x,z)$ along $\mingeodesic$ to get $dk'(y,z) \in T^* _y N.$  In this way we get $$ df'(y) = \int_N dk'(y,z) \Laplacian_{\contact} f(z) \,dz$$ and we find
$$df'(y) - df(y) = \int_N (dk'(y,z) - dk(y,z)) \Laplacian_{\contact} f(z) \,dz$$

Following \cite{kato}, Lemma 1.4, we divide this integral as follows:
Let $\Sigma = B_{2b} (x),$ the ball of radius $2b$ about $x$.
Then $\int_N = \int_{\Sigma} + \int_{N-\Sigma}$. For the integral over $\Sigma$, we have the estimate
$$\int_{\Sigma} dk'(y,z) \Laplacian_{\contact} f(z) \,dz \leq K \lVert\Laplacian_{\contact} f \rVert_{C^0} \int_{\Sigma} \rho(x,z)^{1-2n} \,dz  \leq 2bK \lVert\Laplacian_{\contact} f\rVert_{C^0}.$$
Also $\Sigma \subset B_{3b}(y)$ and
$$ \int_{\Sigma} dk(y,z) \Laplacian_{\theta} f(z)\,dz \leq
K \lVert\Laplacian_{\theta} f \rVert_{C^0} \int_{\Sigma} \rho(y,z)^{1-2n} \,dz 
\leq 3bK \lVert \Laplacian_{\theta} f \rVert_{C^0} $$
Thus 
\begin{equation} \label{sigest} \int_{\Sigma} (dk'(y,z) - dk(y,z)) \Laplacian_{\theta} f(z) \,dz \leq 5bK \lVert \Laplacian_{\theta} f \rVert_{C^0}
\end{equation}

The estimate of the integral over $N-\Sigma$
is more subtle.  If we let $P(\tau)$ denote parallel translation along $\mingeodesic$ from $\mingeodesic(\tau)$ to $\mingeodesic(b)$, so that
$P(\tau): T^*_{\mingeodesic(\tau)} N \rightarrow T^*_y N,$ we find that
\begin{eqnarray}
dk'(y,z) - dk(y,z)& = &P(b)(d k(x,z)) -dk(y,z)\\
& = & \int_0^b P(\tau)
\nabla_{\mingeodesic'(\tau)}dk(\mingeodesic(\tau),z)dz \nonumber
\end{eqnarray}
Therefore
\begin{eqnarray}
 |dk'(y,z) -dk(y,z)| & \leq & \max_{0\leq \tau \leq b} \{|\nabla_{\mingeodesic'(\tau)} dk(\mingeodesic(\tau),z)|\} \\
 & \leq & Kb(\rho(x,z)-b)^{-2n} \nonumber
\end{eqnarray}

We note that $R= {\rm diam}(N).$  Hence
\begin{eqnarray}
\int_{N-\Sigma} |dk'(y,z)-dk(y,z)| & \leq & Kb\int_{N-\Sigma} (\rho(x,z)-b)^{-2n}dz\\
& \leq & Kb \int_{2b}^R (r-b)^{-1}dr \nonumber \\
& = & Kb(\log(R-b) - \log b) \nonumber \\
& \leq & Kb \log(R/b) \nonumber
\end{eqnarray}

Combining this with \eqref{sigest} we find
$$
\int_N |dk'(y,z) -dk(y,z)| \,dz \leq K(5\rho(x,y) + \rho(x,y) \log(R/\rho(x,y))
$$
for $\rho(x,y) < \delta.$  By increasing $K$ we simplify this inequality to
$$
\int_N|dk'(y,z)-dk(y,z) | \,dz \leq K \rho(x,y)(1+ \log(R/\rho(x,y)).
$$
Thus we find that if $\rho(x,y)<\delta,$
\begin{equation} \label{quasi}
|df'((y) - df(y)| \leq K\rho(1+\log(R/\rho(x,y))
\end{equation}
where $K$ is independent of $t$ as before. With this inequality we say that $df$ is quasi-Lipschitz.  Also since $df(x)$ is bounded independently of $x$ and $t$ we find that by further increasing $K$
we get \eqref{quasi} for all $x,y \in N;$ that is, we can drop the restriction $\rho(x,y) < \delta.$

Since $u = S_{\theta} f$, our bound for $df$ gives the same quasi-Lipschitz bound for $u$, again uniformly in $t.$  Using this bound we can find a positive $\alpha$ for which the flow $\eta(t)$ of $u$ is $C^{\alpha}.$ However we will need to show that $\eta(t)^{-1}$ is $C^{\alpha}$.  Fortunately we can do this by the same method.

Given $t_0 \in [0,T_e)$ we define a time dependent vector field $v$ on
$M$ by $v(t) = -u(t_0-t)$. Then if $\sigma$ is the flow of $v$, it is easy to see that the maps 
$\sigma\big(t, \eta(t_0,x)\big)$ and $\eta(t_0-t,x)$ satisfy the same differential equation, and since $\sigma\big(0, \eta(t_0,x)\big) = \eta(t_0,x)$, they must be equal for all times $t\in [0,t_0]$. 
Hence in particular $\sigma(t_0)=\eta(t_0)^{-1}.$ 
%
We proceed to show that $\sigma(t_0)$ is $C^{\alpha}$.

Fix $x$ and $y$ in $M$. Let $\mingeodesic\colon[0,t_0]\times [0,1]$ be the map such that $\tau\mapsto \mingeodesic(t,\tau)$ is the minimal geodesic between $\sigma(t,x)$ and $\sigma(t,y)$ with $\mingeodesic(t,0)=\sigma(t,x)$ and $\mingeodesic(t,1)=\sigma(t,y)$. 
Define 
$$
\phi(t) = \rho(\sigma(t,x), \sigma(t,y)) = \int_0^1 \big\lvert \tfrac{\partial \mingeodesic}{\partial \tau}(t,\tau)\big\rvert \,d\tau.
$$
Then 
\begin{align*}
\phi'(t) &= \frac{d}{dt}\left(\int_0^1 \big\lvert \tfrac{\partial \mingeodesic}{\partial \tau}(t,\tau)\big\rvert \, d\tau\right) \\
&= \int_0^1 \frac{1}{\lvert \frac{\partial \mingeodesic}{\partial \tau}(t,\tau)\rvert} \, \big\langle
\tfrac{\partial\mingeodesic}{\partial \tau}(t,\tau), \tfrac{D}{\partial t} \tfrac{\partial \mingeodesic}{\partial \tau}(t,\tau)\big\rangle \, d\tau
\end{align*}
But 
$\tfrac{D}{\partial t} \tfrac{\partial\mingeodesic}{\partial \tau} = \tfrac{D}{\partial \tau} \tfrac{\partial\mingeodesic}{\partial t}$ by general properties of surface maps, and $\tfrac{D}{\partial \tau} \tfrac{\partial\mingeodesic}{\partial \tau} = 0$ since each $\tau\mapsto \mingeodesic(t,\tau)$ is a geodesic, and thus an integration by parts yields
$$
\phi'(t) = \frac{1}{\lvert \tfrac{\partial \mingeodesic}{\partial \tau}(t,\tau)\rvert} \big\langle \tfrac{\partial \mingeodesic}{\partial \tau}(t,\tau), \tfrac{\partial \mingeodesic}{\partial t}(t,\tau)\big\rangle\big\vert_{\tau=0}^{\tau=1},
$$
using the fact that $\lvert \tfrac{\partial\mingeodesic}{\partial\tau}(t,\tau)\rvert$ is constant in $\tau$ since $\mingeodesic$ is a geodesic in $\tau$. 

Now $\frac{\partial \mingeodesic}{\partial \tau}$ is parallel along $\mingeodesic$, and since the parallel transport
$P$ from $\mingeodesic(0)$ to $\mingeodesic(1)$ preserves inner products, we have 
$$ \langle \tfrac{\partial \mingeodesic}{\partial \tau}(t,0), v\big(t, \sigma(t,x)\big)\rangle = 
\langle \tfrac{\partial\mingeodesic}{\partial \tau}(t,1), Pv\big(t,\sigma(t,x)\big)\rangle. $$
Thus 
\begin{align*} 
\phi'(t) &= \frac{1}{\lvert \tfrac{\partial \mingeodesic}{\partial \tau}(t,\tau)\rvert} \big\langle \tfrac{\partial \mingeodesic}{\partial \tau}(t,1), v\big(t, \sigma(t,y)\big) - Pv\big(t, \sigma(t,x)\big)\big\rangle \\
&\le \lvert v\big(t, \sigma(t,y)\big)
-Pv\big(t,\sigma(t,x)\big)\rvert.
\end{align*}
%
Now since $v$ like $u$ is quasi-Lipschitz on $[0,t_0]$, we have a constant $K$ such that
\begin{equation}
\label{phidiffineq} \phi'(t) \leq K \phi(t)\left(1+ \log \left(\frac{R}{\phi(t)}\right)\right),
\end{equation}
where the constants $R$ and $K$ do not depend on $t$ or $t_0$. 

We proceed to estimate $\phi(t)$.  Let $\psi(t) = \log(\phi(t)/R)$ so $\psi' = \phi'/\phi.$
Then from \eqref{phidiffineq} we get $\psi' \leq K(1-\psi)$.
Integrating this we find
$$\psi(t) \leq \psi(0)e^{-Kt} + 1-e^{-Kt}.$$
Exponentiating this inequality and noting that $\phi(0) \leq R$ we find that
$$\frac{\phi(t)}{R}\leq \left(\frac{\phi(0)}{R}\right)^{e^{-Kt}}e^{1-e^{-Kt}}
\leq \left(\frac{\phi(0)}{R}\right)^{e^{-KT_e}}e
$$
Thus $\phi(t) \leq R^{1-e^{-KT_e}} e \phi(0)^{e^{-KT_e}}$,
so letting $\alpha = e^{-KT_e}$ and $L = eR^{1-e^{-Kt_e}}$,
we get $\phi(t) \leq L \phi(0)^{\alpha}$
and hence
\begin{equation} \label{rhoineq}
\rho(\sigma(t_0,x),\sigma(t_0,y)) \leq L \rho(x,y)^{\alpha}.
\end{equation}
The constants $L$ and $\alpha$ do not depend on the choice of $t_0$, so the estimate \eqref{rhoineq}
holds for all $t_0\in [0,T_e)$. 
From \eqref{rhoineq} and the equation
$$ \Laplacian_{\contact} f(t) = \Laplacian_{\contact} f(0) \circ \sigma(t)$$
we find that $\Laplacian_{\contact} f(t)$ is bounded uniformly in $t$ in
$C^{\alpha}(M,\mathbb{R})$, and thus by standard elliptic theory we get
$f(t)$ bounded in $C^{2+\alpha}(M,\mathbb{R})$,
from which it follows that $S_{\contact} f$ is bounded in $C^{1+\alpha}(TM)$.

Now we show $f(t)$ to be bounded in the $H^{s+1}$ topology, or equivalently that $\Laplacian_{\contact} f(t)$ is bounded in $H^{s-1}.$
We note that
$$ \partial_t \Laplacian_{\contact} f = -S_{\contact}f (\Laplacian_{\contact} f)
$$
so
$$\frac{d}{dt} \int_M (\Laplacian_{\contact} f)^2 \mu = -\int_M S_{\contact} f (\Laplacian_{\contact} f)^2 \mu = \int_M (\Laplacian_{\contact} f)^2 \diver(S_{\contact} f )\mu = 0
$$
Taking $s-1$ spatial derivatives\footnote{ ~Powers of $\nabla$ are defined in a standard way: for any function $f,$ $\nabla f$ is a section of $TM$, $\nabla^2 f$ is a section of $TM \otimes T^*M$ and for any $k$, $\nabla^k f$ is a section of $TM \otimes (T^*M)^{k-1}$.  Inner products are defined by the induced Riemannian metric.} we get
\begin{align*}
\frac{d}{dt} \int_M |\nabla^{s-1} \Laplacian_{\contact} f |^2 \mu & =
-2 \int_M \langle \nabla_{S_{\contact}f} \nabla^{s-1} \Laplacian_{\contact} f, \nabla^{s-1} \Laplacian_{\contact} f\rangle \mu\\
 &  \qquad\qquad -2 \int_M \langle[\nabla_{S_{\contact} f}, \nabla^{s-1}]\Laplacian_{\contact} f,
\nabla^{s-1} \Laplacian_{\contact} f\rangle \mu,
\end{align*}
where $[ \;,\;]$ denotes the commutator.
The first term on the right vanishes, and for the second we use the standard estimate
$$ \| \nabla^l (hg) - h \nabla^l g\|_{H^0} \leq K \big(\|h\|_{H^l} \|g\|_{C^0} + \|\nabla h \|_{C^0} \|g\|_{H^{l-1}}\big),
$$ 
which can be found in \cite{taylor} Chapter 13, Proposition 3.7.
We let $l=s-1,$ $h = S_{\contact} f$ and $ g = \Laplacian_{\contact} f$
and obtain
$$\int_M \big\langle[\nabla_{S_{\contact} f}, \nabla^{s-1}]\Laplacian_{\contact} f,
\nabla^{s-1} \Laplacian_{\contact} f\big\rangle \mu \leq K \|S_{\contact}f\|_{C^1} \| \Laplacian_{\contact} f \|^2_{H^{s-1}}.
$$

Thus we get
$$
\frac{d}{dt} \|\Laplacian_{\contact} f \|^2_{H^{s-1}} \leq K \|S_{\contact}f\|_{C^1} \| \Laplacian_{\contact} f \|^2_{H^{s-1}}
$$
from which it follows that $\|\Laplacian_{\contact} f \|_{H^{s-1}}$
remains bounded on $[0,T_e),$ since $\lVert S_{\contact}f\rVert_{C^0}$ is bounded uniformly in time. The same reasoning combined with a time reversal gives boundedness on $(-T_b,0].$ This completes the proof.
\end{proof}

\begin{rem}
We also get $C^{\infty}$ geodesics for all time.  If the initial datum is $C^{\infty}$, the geodesic exists in $\Diffsexcontact.$  But for any $s' > s,$ there is also a geodesic in $\Diff^{s'}_q \subset \Diffsexcontact$.  By uniqueness of solutions for ordinary differential equations, the two geodesics must coincide.  Since $s'$ can be arbitrarily large the geodesic must be $C^{\infty}$.
\end{rem}

We have shown in Example \ref{geostrophicsphere} that the geodesic equation on $\Diffexcontact(M)$ reduces, with a certain choice of metric on $M$, to the equation 
$$ (\Laplacian f-\alpha^2 f )_t + \{ f, \Laplacian f\} = 0,$$
which is a simplified version of the quasigeostrophic equation; when $M=\mathbb{R}^2$, this is the $f$-plane approximation. Now we finally want to discuss the real quasigeostrophic equation on $S^2$, which is 
\begin{equation}\label{realquasigeo}
(\Laplacian f - \alpha^2 f)_t + \{f, \Laplacian f - \beta \cos{\theta}\} = 0
\end{equation}
for some parameters $\alpha$ and $\beta$.
Here we use spherical coordinates in which $z=\cos{\theta}$ is the height function on the $2$-sphere. See \cite{kuo, CS, STS} for derivations of this equation. 
The equation \eqref{realquasigeo} can be written in conservation form in terms of the ``potential vorticity'' $\omega = \Laplacian f - \alpha^2 f - \beta \cos{\theta}$ as 
\begin{equation}\label{potentialvorticity}
\omega_t + u(\omega) = 0
\end{equation}
where $u = S_{\contact}f$. Now let us show how to derive this equation in our geometric context as an Euler-Arnold equation.

\begin{rem}\label{realquasigeostrophic}
Consider the Lie algebra $T_{\id}\Diffexcontact(M) \cong \Func_{\Reeb}(M)$, with bracket given by $[S_{\contact}f, S_{\contact}g] = S_{\contact}\{f,g\}$, where $\{\cdot,\cdot\}$ is  the contact bracket defined by \eqref{contactbracket}. Let $\varphi\colon M\to\mathbb{R}$ be an arbitrary smooth function, and construct a central extension structure on the product $T_{\id}\Diffexcontact(M) \times \mathbb{R}$ by the rule 
\begin{equation}\label{CEbracket}
[(S_{\contact}f, a), (S_{\contact}g, b)] = \left(S_{\contact}\{f,g\}, \int_M \varphi \{f,g\} \, d\mu\right).
\end{equation}
It is easy to check that this bracket satisfies antisymmetry and the Jacobi identity, so that it is a genuine Lie bracket on a Lie algebra.

Now define an inner product on $T_{\id}\Diffexcontact(M)\times \mathbb{R}$ by 
\begin{equation}\label{CEmetric}
\llangle (S_{\contact}f, a), (S_{\contact}g, b)\rrangle = \int_M \langle S_{\contact}f, S_{\contact}g\rangle \, d\mu + ab.
\end{equation}
Then it is straightforward to compute that the Euler-Arnold equation for this Lie algebra and inner product is the system
\begin{align*}
\Laplacian_{\contact}f_t + \{f, \Laplacian_{\contact}f\} + a \{ f, \varphi\} &= 0, \\
a_t &= 0.
\end{align*}
In particular if $M=S^3$ with Berger metric as in Example \eqref{geostrophicsphere}, and if $\varphi$ is the height function on the sphere $S^2$, then we can obtain the genuine quasigeostrophic equation on $S^2$ with an appropriate choice of parameter $a$.

The proof of global existence in Theorem \ref{globalgeodesic} can be extended without much difficulty, since everything there depends on having a uniform bound for the potential vorticity $\Laplacian_{\contact}f - \alpha^2 f$, which now gets replaced by the real potential vorticity $\Laplacian_{\contact}f - \alpha^2f - \beta \cos{\theta}$. Obviously a uniform bound on one of these implies a uniform bound on the other, so that all the techniques immediately generalize, and we obtain global existence for solutions of the real quasigeostrophic equation on $S^2$.
\end{rem}

\makeatletter \renewcommand{\@biblabel}[1]{\hfill#1.}\makeatother

\end{document}